\documentclass[12pt, a4paper]{amsart}

\usepackage{amsmath, amsthm, amssymb}
\usepackage{graphicx}
\usepackage{enumerate}
\usepackage[top=2.5cm, bottom=2.5cm, left=2.5cm, right=2.5cm]{geometry}
\usepackage{hyperref}

\newtheorem{definition}{Definition}
\newtheorem{lemma}{\bf Lemma}

\newtheorem{theorem}{\bf Theorem}
\newtheorem{corollary}{\bf Corollary}

\theoremstyle{definition}
\newtheorem{remark}{Remark}{\rm}
\newtheorem{example}[subsection]{Example}{\rm}

\usepackage{algorithm}
\usepackage{algorithmic}
\newcommand{\algorithmicparameters}{\textbf{Parameters:}}
\newcommand{\PARAMETERS}{\item[\algorithmicparameters]}
\newcommand{\point}{\!\!{\bf  .}}

\usepackage{color}

\usepackage[T1]{fontenc}
\makeatletter

	\renewcommand{\@secnumfont}{\bfseries}
    \def\section{\@startsection{section}{1}%
    \z@{.7\linespacing\@plus\linespacing}{.5\linespacing}%
    {\normalfont\bfseries\scshape \centering}}
\makeatother

\usepackage{mathrsfs}

\begin{document}

\title{Nonconvex bundle method with application to a delamination problem}
\author[M.N. Dao, J. Gwinner, D. Noll, and N. Ovcharova]{M.N. Dao$^\dag$, J. Gwinner$^\ast$, D. Noll$^\dag$, and N. Ovcharova$^\ast$}
\thanks{$^\dag$ Institut de Math\'ematiques, Universit\'e de Toulouse, France}
\thanks{$^\ast$ Institute of Mathematics, Department of Aerospace Engineering, Universit\"{a}t der Bundeswehr M\"{u}nchen, Germany}
%

%
\maketitle

\begin{abstract}
Delamination is a typical failure mode of composite materials caused by weak bonding. 
It arises when a crack initiates and propagates under a destructive loading. 
Given the physical law 
characterizing the properties
of the interlayer adhesive between the bonded bodies, 
we consider the problem of computing the propagation of the
crack  front and the stress field along the contact boundary. This leads to a
hemivariational inequality, which after discretization by finite elements
we solve by a 
nonconvex bundle method, where 
upper-$C^1$ criteria have to be minimized.  As this is in contrast with other classes
of mechanical problems with non-monotone friction
laws and in other applied fields, where  criteria are typically lower-$C^1$,
we propose a bundle method suited for both types of nonsmoothness.
We prove its global convergence  in the sense of subsequences and test it on a typical
delamination problem of material sciences.

\vspace{.4cm}\noindent
\end{abstract}

\section{Introduction}
We develop a bundle technique to solve nonconvex
variational problems arising  in contact  mechanics and in other applied fields.  
We are specifically  interested in the
delamination of 
composite structures with an adhesive bonding under destructive loading, a failure mode
which is studied in the material sciences.  
When the properties of the  interlayer adhesive between the bonded bodies 
are given in the form of a physical law relating the normal component of the stress vector  
to the relative displacement between the upper and lower boundaries at the crack tip,
the challenge is to compute the  displacement and stress fields in order to assess the reactive destructive forces
along the contact boundary, as the latter are difficult to measure in situ. 
This leads to minimization  of an energy functional,
where a  specific form of nonsmoothness arises in the
boundary integral at the contact boundary.
After discretization via piecewise linear finite elements using the trapezoidal quadrature rule, this
leads to  a finite-dimensional  
nonsmooth  optimization problem of the form
\begin{eqnarray}
\label{program}
\begin{array}{ll}
\mbox{minimize} & f(x) \\
\mbox{subject to}& Ax \leq b
\end{array}
\end{eqnarray}
where $f$ is locally Lipschitz and neither smooth nor convex. Depending
on the nature of the frictional forces, 
the criterion $f$ may be upper-$C^1$ or lower-$C^1$,  see e.g.  Figure \ref{dellaw}.
As these two classes of nonsmooth functions behave substantially
differently when minimized, we are  forced to
expand on existing bundle strategies
and develop an
algorithm  general enough to encompass  both
types of nonsmoothness. We prove its convergence to a critical point in the sense of subsequences,
and show that it 
provides satisfactory numerical results in a simulation of the double cantilever beam test
\cite{gudladt},
one of the most popular  destructive tests to qualify structural adhesive joints.

The difficulty in nonconvex bundling is to provide a suitable
cutting plane oracle which replaces the no longer available 
convex tangent plane. 
One of the oldest oracles,
discussed already in Mifflin \cite{miflin}, and used in the bundle codes of Lemar\'echal
and Sagastiz\'abal \cite{le-saga1,le-saga2}, or the BT-codes of Zowe \cite{zowe,schramm},
uses the method of {\em downshifted tangents}.  While these authors use linesearch 
with Armijo and Wolfe type conditions, which allows only weak convergence certificates
in the sense that {\em some} accumulation point of the sequence of serious iterates is critical,
we favor proximity control in tandem with a suitable backtracking strategy.
This leads to stronger convergence certificates, where {\em every} accumulation point of the sequence of serious iterates
is critical. For instance,
in \cite{noll,pjo,gabarrou} a strong certificate for 
downshifted
tangents with proximity control  was proved within the class of lower-$C^1$ functions, 
but its validity for
upper-$C^1$ criteria remained open. 
An oracle for upper-$C^1$ functions with
a rigorous convergence theory can be based on
the {\em model approach} of \cite{noll,pjo,flows}, but
the latter is not compatible  with the downshift oracle.

To have two strings to one bow is 
unsatisfactory,  as one could hardly expect practitioners to select their strategy
according to such a  distinction, which might not be
easy to make in practice.  In this work we will 
resolve this impasse and present a
cutting plane oracle based on downshifted tangents, which leads to a 
bundle method with strong convergence certificate 
for both types of nonsmoothness. In its principal components
our method agrees with
existing strategies for downshifted tangents,
like \cite{le-saga1,zowe,Makela,Miettinen}, and could therefore be considered as  a justification
of this technique for a wide class of applications. Differences with existing methods 
occur in the management of the proximity control
parameter, which in our approach has to respect certain rules to assure
convergence to a critical point,  without impeding good practical
performance.

The structure of the paper is as follows. 
Section \ref{C1} gives some preparatory information on 
lower- and upper-$C^1$ functions.
Section  \ref{algorithm} presents the algorithm
and comments on its ingredients. Theoretical tools
needed to prove convergence are presented
and employed in sections \ref{modelconcept} and \ref{oracle}.  Section \ref{main}
gives the main convergence result, while section \ref{practical}
discusses practical aspects of the algorithm.
In section \ref{delamination},   we discuss the
delamination problem, which we solve numerically using our bundle algorithm.

Numerical results for contact problem with adhesion based on the bundle-Newton method of  L. Luk\v{s}an and J. Vl\v{c}ek \cite{Luksan} 
can be found e.g. in the book of Haslinger et al. \cite{Haslinger},  in  \cite{Makela, Miettinen}, 
and in the more recent  \cite{ Czepiel,LeoSt}. 
Mathematical analysis and numerical results for quasistatic delamination problems can be found in 
\cite{Kocvara,Roubicek}.

\section{Lower- and upper-$C^1$ functions}
\label{C1} 
Following Spingarn \cite{spingarn}, a locally Lipschitz function
$f:\mathbb  R^n \to \mathbb R$ is  lower-$C^1$ at $x_0$,
if there exists a compact Hausdorff space $K$, a neighborhood $U$
of $x_0$, and 
a mapping $F:U \times K \to \mathbb R$
such that both $F$ and $D_xF$ are jointly continuous and
\begin{eqnarray}
\label{lower}
f(x) = \max\{F(x,y): y\in K\}
\end{eqnarray}
is satisfied
for $x\in U$. The function $f$ is upper-$C^1$ at $x_0$ if $-f$ is
lower-$C^1$ at $x_0$.  

In a minimization problem (\ref{program}), we expect 
lower- and upper-$C^1$ functions to behave completely  differently.
Minimizing a lower-$C^1$ function ought to lead to  real difficulties, as on
descending we  move {\em into} the zone of nonsmoothness, which for lower-$C^1$ goes downward. 
In contrast,  
upper-$C^1$ functions are generally expected to be well-behaved, as intuitively
on descending we 
move {\em away} from the nonsmoothness, which here goes upward. The present application
shows that this
argument is too simplistic. 
Minimization of upper-$C^1$
functions  leads to real difficulties, which we explain subsequently.
In
delamination for composite materials
we encounter objective functions
of the form
\begin{eqnarray}
\label{structure}
f(x) = f_s(x) + \int_0^1 \min_{i\in I} f_i(x,t)\, dt, 
\end{eqnarray}
where $f_s$ gathers the smooth part, while the integral
term, due to the minimum, is responsible for the nonsmoothness. 

\textcolor{black}{
\begin{lemma}
Suppose $f_s$ is of class $C^1$ and
the $f_i$ are jointly of class $C^1$.
Then the function {\rm (\ref{structure})} is upper-$C^1$ and can be represented in the form
\begin{eqnarray}
\label{struct}
f(x) =f_s(x) + \min_{\sigma\in \Sigma} \int_0^1 f_{\sigma(t)}(x,t)\, dt, 
\end{eqnarray}
where $\Sigma$ is the set of all measurable
mappings $\sigma: [0,1] \to I$.
\end{lemma}
}

\begin{proof}
Let us first prove (\ref{struct}).
For $\sigma\in \Sigma$ and fixed $x\in \mathbb R^n$ the function
$t\mapsto f_{\sigma(t)}(x,t)$ is measurable, and since
$\min_{i\in I} f_i(x,t) \leq f_{\sigma(t)}(x,t)\leq \max_{i\in I} f_i(x,t)$, it is also
integrable. Hence
$F(x,\sigma)=f_s(x)+\int_0^1 f_{\sigma(t)}(x,t)\, dt$ is well-defined, and clearly
$F(x,\sigma)\geq f(x)$, so we have $\inf_{\sigma\in \Sigma} F(x,\sigma) \geq f(x)$. 

To prove the reverse estimate,
fix $x\in \mathbb R^n$ and consider the closed-valued multifunction $\Phi:[0,1] \to 2^I$ defined by 
$\Phi(t) = \{i\in I: f_i(x,t) = \min_{i'\in I}f_{i'}(x,t)\}$. Since the $f_i(x,\cdot)$ are measurable and $I$
is finite,  $\Phi$
is a measurable multifunction. Choose a measurable selection $\sigma$, that is, $\sigma\in \Sigma$
satisfying $\sigma(t) \in \Phi(t)$ for every $t\in [0,1]$. Then clearly $F(x,\sigma)=f(x)$. 
This proves (\ref{struct}).

Let us now show that $f$ is upper-$C^1$.
We consider 
$
\varphi(x,t)=\min_{i\in I} f_i(x,t). 
$
In view of \cite{spingarn} $\varphi(\cdot,t)$  is upper-$C^1$ and its Clarke subdifferential $\partial \varphi(\cdot,t)$
is strictly supermonotone uniformly over $t\in [0,1]$. 
By Theorem 2 in \cite{georgiev},
$\varphi(\cdot,t)$ is approximately concave uniformly over $t\in [0,1]$. Integration with respect to 
$t\in [0,1]$ then yields an approximately concave function with respect to $x$, 
which by the equivalences in \cite{georgiev} and \cite{spingarn} is 
upper-$C^1$. 
\end{proof}

Note that the minimum (\ref{struct})
is  semi-infinite  even though $I$ is finite. Minimization of (\ref{structure})  cannot
be converted into a NLP, as would be possible in the min-max case. 
The representation (\ref{struct}) highlights the difficulty in minimizing (\ref{structure}).
Minimizing a minimum has a 
disjunctive character, 
and due to the large size of $\Sigma$ this could lead to
a combinatorial situation with intrinsic difficulty.

\section{The model concept}
\label{modelconcept}
The model of a nonsmooth function was introduced
in \cite{pjo} and is a key element in understanding the bundle concept.

\begin{definition}
[Compare {\rm \cite{pjo}}] 
A function $\phi:\mathbb R^n\times \mathbb R^n \to \mathbb R$
is called a model of the locally Lipschitz function $f:\mathbb R^n\to \mathbb R$
on the set $\Omega \subset \mathbb R^n$ if the following axioms are
satisfied:
\begin{enumerate}
\item[$(M_1)$] For every $x\in \Omega$ the function
$\phi(\cdot,x):\mathbb R^n\to \mathbb R$ is convex,
$\phi(x,x) = f(x)$ and $\partial_1\phi(x,x) \subset \partial f(x)$.
\item[$(M_2)$] For every $x\in \Omega$ and every $\epsilon >0$ there
exists $\delta > 0$ such that $f(y) \leq \phi(y,x) + \epsilon \|y-x\|$
for every $y\in B(x,\delta)$.
\item[$(M_3)$] The function $\phi$ is jointly upper semicontinuous, i.e.,
$(y_j,x_j)\to (y,x)$ on $ \mathbb R^n\times \Omega$ implies $\displaystyle\limsup_{j\to\infty} \phi(y_j,x_j)\leq \phi(y,x)$.
\hfill $\square$
\end{enumerate}
\end{definition}
We recall that every locally Lipschitz function $f$ has the so-called {\em standard model}
\[
\phi^\sharp(y,x) = f(x) + f^0(x,y-x),
\]
where $f^0(x,d)$ is the Clarke directional derivative
of $f$ at $x$ in direction $d$. The same function $f$ may in general have several models $\phi$, and
following \cite{noll,flows}, the standard
$\phi^\sharp$ is the smallest one. Every
model $\phi$ gives rise to a bundle strategy. The question is then whether this bundle strategy
is successful. This depends on 
the following property of $\phi$.

\begin{definition}
A model  $\phi$  of $f$ on $\Omega$
is said to be strict at $x_0\in \Omega$ if axiom $(M_2)$ is replaced
by the stronger
\begin{enumerate}
\item[$(\widehat{M}_2)$] For every $\epsilon > 0$ there exists $\delta > 0$
such that $f(y) \leq \phi(y,x) + \epsilon \|y-x\|$ for all $x,y\in B(x_0,\delta)$.
\end{enumerate}
We say that $\phi$ is a strict model on $\Omega$, if it is strict at every
$x_0\in \Omega$. \hfill $\square$
\end{definition}

\begin{remark}
We may write axiom $(M_2)$ in the form
$f(y) \leq \phi(y,x_0) + {\rm o}(\|y-x_0\|)$ for $y\to x_0$, and
$(\widehat{M}_2)$ as $f(y) \leq \phi(y,x) + {\rm o}(\|y-x\|)$ 
for $x,y\to x_0$. Except for the fact that these concepts are
one-sided, this  is precisely  the difference between differentiability
and strict differentiability. Hence the nomenclature.
\end{remark}

\begin{lemma}
[Compare {\rm \cite{noll,flows}}] 
Suppose $f$ is  upper-$C^1$. Then its standard
model $\phi^\sharp$ is strict, and hence every model $\phi$ of $f$ is strict.
\hfill $\square$
\end{lemma}

\begin{remark}
For convex $f$
the standard model $\phi^\sharp$ is in general not strict, but $f$ may be used as its own model $\phi(\cdot,x)=f$.  
For nonconvex $f$, a wide range of
applications is covered by composite
functions $f = g \circ F$ with $g$ convex and $F$
differentiable. Here the so-called natural
model $\phi(y,x) = g(F(x) + F'(x)(y-x))$ can be used, because it is strict
as soon as $F$ is class $C^1$. This includes 
lower-$C^2$ functions in the sense of \cite{rock}, lower-$C^{1,\alpha}$
functions in the sense of \cite{malick-dani},
or amenable functions in the sense of \cite{amenable},
which allow representations of the form $f=g\circ F$
with $F$ of class $C^{1,1}$.
\end{remark}

We conclude with the remark that lower-$C^1$
functions also admit strict models, even though in
that case the construction is more delicate. The strict model
in that case cannot be exploited algorithmically,
and for lower-$C^1$ functions we prefer the oracle
concept, which will be discussed in section
\ref{oracle}.

\section{Elements of the algorithm}
\label{algorithm}
In this section we briefly explain the main features of the algorithm.
This concerns building the working model, computing 
the solution of the tangent program, checking acceptance, updating the working model after null steps,
and the management of the proximity control parameter.

\subsection{Working model}
At the current serious iterate  $x$  the  inner loop
of the algorithm at counter $k$ computes an approximation $\phi_k(\cdot,x)$ of
$f$ in a neighborhood of $x$, called a first-order
working model. The working model is a polyhedral
convex function of the form
\begin{eqnarray}
\label{working}
\phi_k(\cdot,x) = \max_{(a,g)\in \mathcal G_k} a + g^\top (\cdot - x),
\end{eqnarray}
where $\mathcal G_k$ is a finite set of affine functions
$y\mapsto a + g^\top(y-x)$ satisfying $a\leq f(x)$,
referred to as {\em planes}. The set $\mathcal G_k$ is  updated
during the inner loop $k$. 
At each step $k$ the following
rules have to be respected when updating
$\mathcal G_{k}$ into $\mathcal G_{k+1}$:
\begin{enumerate}
\item[$(R_1)$] One or several cutting planes at the null step $y^k$, 
generated by an abstract  cutting plane oracle,
are added to $\mathcal G_{k+1}$.
\item[($R_2)$] The so-called aggregate plane $(a^*,g^*)$, which consists of convex
combinations of elements of $\mathcal G_{k}$, is added to $\mathcal G_{k+1}$.
\item[$(R_3)$] Some older  planes in $\mathcal G_{k}$, which become obsolete
through the addition of the aggregate
plane, are discarded and not kept in $\mathcal G_{k+1}$. 
\item[$(R_4)$] Every $\mathcal G_k$ contains at least one so-called exactness plane  $(a_0,g_0)$, where
exactness plane means $a_0=f(x)$, $g_0\in\partial f(x)$. This assures
$\phi_k(x,x)=f(x)$, hence the name.
\item[$(R_5)$] We have to make sure that each working model $\phi_k$
satisfies $\partial_1 \phi_k(x,x) \subset \partial f(x)$. 
\end{enumerate}
Once the first-order working model $\phi_k(\cdot,x)$
has been built, the second-order working model
$\Phi_k(\cdot,x)$ is of the form
\begin{eqnarray}
\label{second}
\Phi_k(\cdot,x)=\phi_k(\cdot,x)
+
\textstyle\frac{1}{2}(\cdot - x)^\top Q(x) (\cdot -x),
\end{eqnarray}
where $Q(x)=Q(x)^\top$ is a possibly indefinite symmetric 
matrix, depending only on the current serious iterate $x$, and fixed
during the inner loop $k$.  The second-order term includes curvature information on $f$,
if available.

\subsection{Tangent program and acceptance test}
Once the second-order working model (\ref{second})
is formed and the proximity control parameter
$\tau_{k-1} \to \tau_k$ is updated, we solve the tangent program
\begin{eqnarray}
\label{tangent}
\begin{array}{ll}
\mbox{minimize} & \Phi_k(y,x) + \frac{\tau_k}{2}\|y-x\|^2 \\
\mbox{subject to} & Ay\leq b
\end{array}
\end{eqnarray}
Here the proximity control parameter $\tau_k$
satisfies $Q + \tau_k I \succ 0$, which assures that
(\ref{tangent}) is strictly convex and has a unique solution,
$y^k$, called the {\em trial step}. 
The trial step
is a candidate to become the new serious iterate $x^+$.
In order to decide whether $y^k$ is acceptable, we compute the test
\begin{eqnarray}
\label{rho}
\rho_k=\frac{f(x)-f(y^k)}{f(x) - \Phi_k(y^k,x)} \stackrel{?}{\geq} \gamma,
\end{eqnarray}
where $0 < \gamma < 1$ is some fixed parameter.
If $\rho_k \geq \gamma$, then $x^+ = y^k$
is accepted and called a {\em serious step}. In this case the inner loop ends successfully. 
On the other hand, if $\rho_k < \gamma$, then $y^k$
is rejected and called a {\em null step}. In this case the inner loop $k$
continues. This means we will update  working model
$\Phi_k(\cdot,x) \to \Phi_{k+1}(\cdot,x)$, adjust the proximity control parameter $\tau_k\to \tau_{k+1}$, 
and solve (\ref{tangent}) again.

Note that
the test (\ref{rho}) corresponds to the usual
Armijo descent condition used in linesearches, or to the standard
acceptance test in trust region methods.

\subsection{Updating the working model via aggregation}
Suppose the trial step $y^k$ fails the acceptance test (\ref{rho})
and is declared a null step. Then the inner loop has to continue,
and we have to improve the working model at the next sweep
in order to perform better. Since the second-order part of the working model
$\frac{1}{2}(\cdot - x)^\top Q(x)(\cdot - x)$
remains invariant, we will update the first-order part only.

Concerning rule $(R_2)$,
by the necessary optimality condition for (\ref{tangent}),
there exists a multiplier $\eta^*$ such that
\[
0\in \partial_1 \Phi_k(y^k,x) + \tau_k(y^k-x) + A^\top \eta^*,
\]
or what is the same,
\[
(Q(x)+\tau_k I)(y^k-x) - A^\top \eta^* \in \partial_1 \phi_k(y^k,x).
\]
Since $\phi_k(\cdot,x)$ is by construction a maximum of 	affine
planes,  we use the standard description of the 
convex subdifferential of a max-function. Writing $\mathcal G_k
=\{(a_0,g_0),\dots,(a_p,g_p)\}$ for  $p ={\rm card}(\mathcal G_k)+1$,
we find non-negative multipliers $\lambda_0,\dots,\lambda_p$
summing up to 1 such that
\[
(Q(x)+\tau_kI)(y^k-x)-A^\top\eta^* = \sum_{i=0}^p \lambda_i g_i,
\]
and in addition, $a_i + g_i^\top (y^k-x) = \phi_k(y^k,x)$ for all
$i\in \{0,\dots,p\}$ with $\lambda_i > 0$. We say that those
planes which are active at $y^k$ are  {\em called by the aggregate plane}.
In the above rule $(R_3)$ we allow those to be removed  from
$\mathcal G_{k}$.
We now define the aggregate plane as:
\[
a^*_k = \sum_{i=0}^p \lambda_i a_i, \quad g^*_k = \sum_{i=0}^p \lambda_i g_i.
\]
Note that by construction the aggregate plane
$m_k^*(\cdot,x)=a^*_k + g^{*\top}_k(\cdot - x)$ at null step $y^k$ satisfies
$m_k^*(y^k,x)=a^*+g^{*\top}(y^k-x)=\phi_k(y^k,x)$. This construction is standard
and follows the original idea in Kiwiel \cite{kiwiel_aggregate}. It assures in particular that
$\Phi_{k+1}(y^k,x) \geq m_k^*(y^k,x) + \frac{1}{2}(y^k-x)^\top Q(x)(y^k-x) = \Phi_k(y^k,x)$.

\subsection{Updating the working model by cutting planes and exactness planes}
The crucial improvement in the first-order working model is in adding a
cutting plane which cuts away the unsuccessful trial step $y^k$ according to rule $(R_1)$. We shall
denote the cutting plane as $m_k(\cdot,x) = a_k + g^\top(\cdot - x)$. The only requirement
for the time being is that $a_k \leq f(x)$, as this assures $\phi_{k+1}(x,x)\leq f(x)$.
Since we also maintain at least one exactness plane of the form
$m_0(\cdot,x)=f(x) + g_0^\top(\cdot - x)$ with $g_0\in \partial f(x)$, we assure
$\phi_{k+1}(x,x)= \Phi_{k+1}(x,x)=f(x)$.  Later we will also have to check the validity of $(R_5)$.

It is possible to integrate
so-called {\em anticipated cutting planes} in the new working model $\mathcal G_{k+1}$. 
Here anticipated designates all planes which are not 
based on the rules exactness, aggregation, cutting planes. Naturally, 
adding such planes can not be allowed in an arbitrary way, because axioms $(R_1) - (R_5)$ have to be respected.

\begin{remark}
It may be beneficial to choose a new exactness plane $m_0(\cdot,x)= f(x) + g^\top (\cdot-x)$
after each null step $y$, namely the one which satisfies $m_0(y,x) = f^0(x,y-x)$.  
If $x$ is a point of differentiability of $f$, then all these exactness planes are identical anyway, so no
extra work occurs. On the other hand, computing $g\in \partial f(x)$
such that $g^\top (y-x) = f^0(x,y-x)$ is usually cheap. Consider for instance 
eigenvalue optimization, where $f(x) = \lambda_1\left( F(x)\right)$, $x\in \mathbb R^n$,
$F:\mathbb R^n \to \mathbb S^m$, and $\lambda_1:\mathbb S^m\to \mathbb R$ is
the maximum eigenvalue function of $\mathbb S^m$. Then
$f^0(x,d) = \lambda_1'(X,D) = \lambda_1(Q^\top DQ)$, where $X=F(x)$,  $D=F'(x)d$, and
where $Q$ is a $t \times m$ matrix whose columns form an orthogonal basis
of the maximum eigenspace of $X$ of dimension $t$ \cite{cullum}. Then $G=QQ^\top \in \partial \lambda_1(X)$ attains
$\lambda_1'(X,D)$,  hence
$g = F'(x)^* QQ^\top$ attains $f'(x,d)$. Since usually $t\ll m$, the computation of $g$ is cheap.
\end{remark}

\subsection{Management of proximity control}
The central novelty of the bundle methods developed in
\cite{noll,pjo,anp} is the discovery that in the absence of convexity the proximity
control parameter $\tau$ has to follow certain basic rules
to assure convergence of the sequence $x^j$
of serious iterates. This is in contrast with convex bundle
methods, where 
$\tau$ could in principle be frozen once and for all.   
More precisely, suppose $\phi_k(\cdot,x)$
has failed and produced only a null step $y^k$. Having built the new
model $\phi_{k+1}(\cdot,x)$, 
we compute the secondary
test 
\begin{eqnarray}
\label{tilde-rho}
\widetilde{\rho}_k= \frac{f(x)-\Phi_{k+1}(y^k,x)}{f(x)-\Phi_k(y^k,x)}\stackrel{?}{\geq} \widetilde{\gamma},
\end{eqnarray}
where $0 < \gamma < \widetilde{\gamma}<1$ is fixed. Our decision is
\begin{eqnarray}
\label{dec}
\tau_{k+1}=
\left\{
\begin{array}{ll}
2\tau_k &\mbox{if } \widetilde{\rho}_k \geq \widetilde{\gamma} \\
\tau_k & \mbox{if } \widetilde{\rho}_k < \widetilde{\gamma}
\end{array}
\right.
\end{eqnarray}
The rationale of (\ref{tilde-rho})  is to decide whether improving
the model by adding planes will suffice, or 
shorter steps have to be forced by increasing $\tau$. 

The denominator in (\ref{tilde-rho})
gives the model predicted progress $f(x)-\phi_k(y^k,x)=\phi_k(x,x)-\phi_k(y^k,x)>0$
at $y^k$. On the other hand, the numerator $f(x)-\phi_{k+1}(y^k,x)$
gives the progress over $x$ we would achieve at $y^k$, had we already known
the cutting planes drawn at $y^k$. Due to aggregation we know
that $\phi_{k+1}(y^k,x) \geq \phi_k(y^k,x)$, so that
$\widetilde{\rho}_k \leq 1$, but values
$\widetilde{\rho}_k \approx 1$ indicate that little to no
progress is achieved by adding the cutting plane.
In this case we decide that the $\tau$-parameter must be increased to
force smaller steps, because that reinforces the  agreement
between $f$ and $\phi_{k+1}(\cdot,x)$. 

In the test (\ref{dec}) we replace $\widetilde{\rho}_k\approx 1$ by
$\widetilde{\rho}_k\geq \widetilde{\gamma}$ for some fixed $0 < \gamma < \widetilde{\gamma} < 1$. 
If $\widetilde{\rho}_k < \widetilde{\gamma}$, 
then the quotient if far from 1 and we decide that adding planes has still the potential
to improve the situation. In that event we do not increase
$\tau$. 

Let us next consider the management
of $\tau$ in the outer loop. Since $\tau$ can only increase or 
stay fixed in the inner loop, we allow $\tau$ to decrease
between serious steps $x\to x^+$, respectively, $x^j \to x^{j+1}$.
This is achieved by the test
\begin{eqnarray}
\label{Gamma}
{\rho}_{k_j} =\frac{f(x^j) - f(x^{j+1})}{f(x^j) - \Phi_{k_j}(x^{j+1},x^j)}
\stackrel{?}{\geq} \Gamma,
\end{eqnarray}
where $0 < \gamma \leq \Gamma < 1$ is fixed. In other words,
if at acceptance we have not only $\rho_{k_j}\geq \gamma$,
but even $\rho_{k_j} \geq \Gamma$, then we decrease
$\tau$ at the beginning of the next inner loop $j+1$, because
we may trust the model. On the other hand, if
$\gamma \leq \rho_{k_j} < \Gamma$ at acceptance, then
we memorize the last $\tau$-parameter used, that is $\tau_{k_j}$ at the end of the 
$j$th inner loop.

\begin{remark}
We should compare our management of the proximity control
parameter $\tau$ with other strategies in the literature. For instance
M\"akel\"a {\em et al.} \cite{Makela} consider a very
different management of $\tau$, which is motivated by
the convex case. 
\end{remark}

\subsection{Statement of the algorithm}
We are now ready to give our formal statement
of algorithm \ref{algo1}.


\begin{algorithm}
\caption{{\Large $\point$} Proximity control algorithm for (\ref{program}).}\label{algo1}
\begin{algorithmic}[1]
\hrule
\vskip 2pt
\PARAMETERS $0<\gamma < \Gamma < 1$, $\gamma < \widetilde{\gamma}<1$, 
$0< q <  \infty$,
 $q <  T  <  \infty$.

\STATE {\bf Initialize outer loop}. Choose 
initial guess $x^1$ with $Ax^1 \leq b$ and an initial matrix $Q_1=Q_1^\top$
with $-qI\preceq Q_1 \preceq qI$.  
Fix memory control parameter $\tau_1^\sharp$ such that $Q_1+\tau_1^\sharp I\succ 0$.
Put $j=1$.
\STATE {\bf Stopping test}. At outer loop counter $j$, stop if $0\in \partial f(x^j) + A^\top \eta^*$
for some multiplier $\eta^*\geq 0$.
Otherwise goto inner loop.
\STATE {\bf Initialize inner loop}. Put inner loop counter $k=1$ and initialize
$\tau$-parameter using the memory element, i.e., 
$\tau_1 = \tau^\sharp_j$. Choose initial convex working model $\phi_1(\cdot,x^j)$,
possibly recycling some planes from previous sweep $j-1$, 
and let $\Phi_1(\cdot,x^j)=\phi_1(\cdot,x^j)+\frac{1}{2}(\cdot-x^j)^\top Q_j(\cdot-x^j)$. 
\STATE {\bf Trial step generation.} At inner loop counter $k$ solve tangent program
\[
\min_{Ay\leq b} \Phi_k(y,x^j)+\textstyle\frac{\tau_k}{2}\|y-x^j\|^2.
\]
The solution is the new trial step $y^{k}$.
\STATE {\bf Acceptance test.}
Check whether 
\[
\rho_k=\frac{f(x^j)-f(y^{k})}{f(x^j)-\Phi_k(y^{k},x^j)} \ge \gamma.
\]
If this is the case put $x^{j+1}=y^{k}$ (serious step), quit inner loop and goto step 8. 
If
this is not the case (null step) continue inner loop with step 6.
\STATE {\bf Update working model}. 
Build new convex working model $\phi_{k+1}(\cdot,x^j)$ based
on null step $y^{k}$
by adding an exactness plane $m_k^\sharp(\cdot,x^j)$ satisfying $m_k^\sharp(y^k,x^j)=f^0(x^j,y^k-x^j)$,
a downshifted tangent $m_k^\downarrow(\cdot,x^j)$,  and the  aggregate plane $m_k^*(\cdot,x^j)$. 
Apply rule $(R_3)$ to avoid overflow. Build $\Phi_{k+1}(\cdot,x^j)$,
and goto step 7.
\STATE {\bf Update proximity parameter}. Compute
\[
\widetilde{\rho}_k = \frac{f(x^j)-\Phi_{k+1}(y^{k},x^j)}{f(x^j)-\Phi_k(y^{k},x^j)}.
\]
Put 
\[
\tau_{k+1}= \left\{
\begin{array}{ll}
\vspace{.1cm}
\tau_k, &\mbox{if } \widetilde{\rho}_k  < \widetilde{\gamma} \qquad\mbox{(bad)} \\
\vspace{.1cm}
2\tau_k, &\mbox{if } \widetilde{\rho}_k \ge \widetilde{\gamma}\qquad\mbox{(too bad)}\end{array}
\right.
\]
Then increase counter $k$ and continue inner loop with step 4.
\STATE {\bf Update $Q_j$ and memory element}. 
Update matrix $Q_j\to Q_{j+1}$,  respecting $Q_{j+1}=Q_{j+1}^\top$ and
$-qI\preceq Q_{j+1}\preceq qI$. Then
store new memory element
$$\tau_{j+1}^\sharp = \left\{
\begin{array}{ll}
\tau_{k}, & \mbox{if } \gamma \leq \rho_k < \Gamma  \qquad\mbox{    (not bad)}\\
&\\
\displaystyle\textstyle\frac{{1}}{2}\tau_{k}, &\mbox{if } \rho_k \ge \Gamma \qquad\qquad \mbox{    (good)}
\end{array}
\right.$$   
 Increase $\tau_{j+1}^\sharp$ if necessary to ensure $Q_{j+1}+\tau_{j+1}^\sharp I\succ 0$. 
 \STATE {\bf Large multiplier safeguard rule.}
 If  $\tau_{j+1}^\sharp > T$ then re-set $\tau_{j+1}^\sharp = T$.
 Increase outer loop counter $j$  by 1 and loop back to step 2.
\end{algorithmic}
\hrule
\end{algorithm}


\section{Nonconvex cutting plane oracles}
\label{oracle}
In the convex cutting plane method \cite{rudzsinski,hule:93}
unsuccessful trial steps $y^k$ are cut away by adding
a tangent plane to $f$ at $y^k$ into the model.
Due to convexity, the cutting plane is below
$f$ and can therefore be used to
construct an approximation
(\ref{working}) of $f$. For nonconvex $f$, cutting planes are more
difficult to construct, but several ideas have been
discussed. We mention \cite{saga,miflin}.
In \cite{noll} we have proposed an axiomatic approach,
which has the advantage that it covers the applications
we are aware of, and allows a convenient convergence theory.
Here we  use this axiomatic approach in
the convergence proof.

\begin{definition}
[Compare {\rm \cite{noll}}] Let 
$f$ be locally Lipschitz. A cutting plane oracle 
for $f$ on the set $\Omega$ is an operator $\mathscr O$ which,
with every pair $(x,y)$, $x$ a serious iterate in $\Omega$, $y\in\mathbb R^n$ a null step, 
associates an affine function
$m_{y}(\cdot,x)= a + g^\top (\cdot - x)$, called the cutting plane at null step $y$
for serious iterate $x$, so that the following axioms are satisfied:
\begin{enumerate}
\item[$(O_1)$] For $y=x$ we have $a=f(x)$ and $g\in \partial f(x)$.
\item[$(O_2)$] Let $y_j\to x$. Then
there exist $\epsilon_j \to 0^+$ such that $f(y_j)\leq m_{y_j}(y_j,x)+\epsilon_j\|y_j-x\|$.
\item[$(O_3)$] Let $x_j\to x$ and  $y_j,y_j^+\to y$. Then there exists $z\in \mathbb R^n$  
such that\\
$\displaystyle\limsup_{j\to\infty} m_{y_j^+}(y_j,x_j)\leq m_{z}(y,x)$.
\hfill $\square$
\end{enumerate}
\end{definition}

As we shall see, these axioms are aligned with
the model axioms $(M_1) - (M_3)$. Not unexpectedly, there is also
a strict version of $(O_2)$.

\begin{definition}
\label{strict-oracle}
A cutting plane oracle $\mathscr O$  for $f$ is called strict
at $x_0$ if the following strict version of $(O_2)$
is satisfied:
\begin{enumerate}
\item[$(\widehat{O}_2)$]
Suppose $y_j,x_j \to x$. Then there exist $\epsilon_j\to 0^+$
such that $f(y_j)\leq m_{y_j}(y_j,x_j)+ \epsilon_j \|y_j-x_j\|$.  \hfill $\square$
\end{enumerate}
\end{definition}

We now discuss two
versions of the oracle 
which are of special interest for our applications. 

\begin{example}
[Model-based oracle] 
Suppose $\phi$ is a model
of $f$. Then we can generate a cutting plane
for serious iterate $x$ and trial step $y$ by taking
$g\in \partial_1 \phi(y,x)$ and putting
\[
m_{y}(\cdot,x) = \phi(y,x) + g^\top (\cdot - y)
= \phi(y,x) + g^\top (x-y) + g^\top( \cdot - x).
\]
Oracles generated by a model $\phi$ 
in this way will be denoted $\mathscr O_\phi$. Note that $\mathscr O_\phi$
coincides with the standard oracle if $f$ is convex and $\phi(\cdot,x)=f$,
i.e., if the convex $f$ is chosen as its own model. In more general cases, the simple idea of this oracle
is that in the absence of convexity,  where tangents to $f$ at $y$
are not useful, we simply take tangents
of $\phi(\cdot,x)$ at $y$. Note that the model-based oracle $\mathscr O_\phi$
is strict as soon as the model $\phi$ is strict.
\hfill $\square$
\end{example}

\begin{example}
[Standard oracle] A special case of the model-based oracle
is obtained by choosing the standard model $\phi^\sharp$.
Due to its significance for our present work we call this the standard oracle.
The standard cutting plane for serious step
$x$ and null step $y$ is
$m_y^\sharp(\cdot,x) = f(x)+ g^\top (\cdot -x)$, where the Clarke
subgradient $g\in \partial f(x)$ is  one of those that satisfy $g^\top (y-x)=f^0(x,y-x)$.
The standard oracle is strict iff $\phi^\sharp$ is strict. As was 
observed before, this is for instance the case  when $f$ is upper-$C^1$.
Note a specificity of the standard oracle: every standard cutting plane $m_y^\sharp(\cdot,x)$
is also an exactness plane at $x$. 
\hfill $\square$
\end{example}

\begin{example}
[Downshifted tangents] Probably the oldest oracle used
for nonconvex functions are downshifted tangents, which we
define as follows. For serious iterate $x$ and null step $y$
let $t(\cdot) = f(y) + g^\top (\cdot - y)$ be a tangent of $f$ at $y$.
That is, $g\in \partial f(y)$. Then we shift $t(\cdot)$ down
until it becomes useful for the model (\ref{working}). Fixing a parameter
$c > 0$, this is organized as follows: We define the cutting plane as
$m_{y}^\downarrow(\cdot,x) = t(\cdot) - s$, where the downshift $s \geq 0$
satisfies
\[
s = 
[t(x) - f(x) + c\|y-x\|^2]_+ .
\]
In other words,
$m_{y}^\downarrow(\cdot,x) = a + g^\top (\cdot-x)$, where
$a = \min\{ t(x), f(x)-c\|y-x\|^2  \}$. Note that this procedure
aways satisfies axioms $(O_1)$ and $(O_3)$, whereas axioms
$(O_2)$, respectively, $(\widehat{O}_2)$, are satisfied
if $f$ is lower-$C^1$ at $x_0$. In other words, see \cite{noll}, for $f$ lower-$C^1$
this is an oracle, which is automatically strict.
\hfill $\square$
\end{example}

Motivated by the previous examples, we now define
an oracle which works for both lower-$C^1$  and  upper-$C^1$.

\begin{example}
[Modified downshift] \label{modified}
Let $x$ be the current serious iterate, $y$ a null  step
in the inner loop
belonging to $x$. Then we form the downshifted tangent $m_y^\downarrow(\cdot,x):=t(\cdot)-s$,
that is, the cutting plane we would get from the downshift oracle,
and we form the standard oracle plane $m_y^\sharp(\cdot,x)=f(x) + g^\top (\cdot - x)$, where
the Clarke subgradient $g$ satisfies $f^0(x,y-x) = g^\top(y-x)$. Then we define
\[
m_y(\cdot,x) = \left\{
\begin{array}{ll}
m_y^\downarrow(\cdot,x) &\mbox{ if } m_y^\downarrow(y,x) \geq m_y^\sharp(y,x) 
\vspace*{.2cm}\\
m_y^\sharp(\cdot,x) &\mbox{ else}
\end{array}
\right.
\]
In other words, among the two candidate cutting planes $m_y^\downarrow(\cdot,x)$
and $m_y^\sharp(\cdot,x)$, we take the one which has the larger value
at the null step $y$. 

Note that this is the oracle we use in our algorithm.  Theorem \ref{theorem1}
clarifies when this oracle is strict.
\hfill $\square$
\end{example}

Given an operator $\mathscr O$ which
with every pair $(x,y)$ of serious step $x$ and null step $y$ associates a cutting
plane $m_y(\cdot,x)=a+g^\top (\cdot - x)$, we fix a constant $M>0$ and define
what we call   the upper envelope function of the oracle
\[
\phi^\uparrow(\cdot,x) = \sup\{ m_y(\cdot,x):  \|y-x\| \leq M\}.
\]
The crucial property of $\phi^\uparrow$
is the following
\begin{lemma}
Suppose $\mathscr O: (x,y) \mapsto m_y(\cdot,x)$ is a cutting plane oracle
satisfying axioms $(O_1)-(O_3)$. Then $\phi^\uparrow$ is a model of $f$. Moreover,
if the oracle satisfies $(\widehat{O}_2)$, then $\phi^\uparrow$ is strict.
\hfill $\square$
\end{lemma} 

The proof can be found in \cite{noll}. We refer to $\phi^\uparrow$ as the upper envelope
model associated with the oracle $\mathscr O$. Since in turn every model $\phi$ gives rise to a model-based oracle,
$\mathscr O_\phi$,
it follows that having a strict oracle and having a strict model are equivalent properties of $f$.
Note, however, that the model $\phi^\uparrow$ is in general not practically useful. It is a theoretical tool in the convergence proof.

\begin{remark}
If we start with a model $\phi$, then build $\mathscr O_\phi$, and go back to
$\phi^\uparrow$,  we get back to $\phi$, at least locally.

On the other hand, going from an oracle $\mathscr O$ to its envelope model $\phi^\uparrow$,
and then back to the model based oracle $\mathscr O_{\phi^\uparrow}$ does {\em not}
necessarily lead back to the oracle $\mathscr O$.
\end{remark}

We are now in the position to check axiom $(R_5)$.

\begin{corollary}
All working models $\phi_k$ constructed in our algorithm
satisfy $\partial_1\phi_k(x,x)\subset \partial f(x)$. \hfill $\square$
\end{corollary}


\section{Main convergence result}
\label{main}
In this section we state and prove the main result
of this work and give several
consequences.

\begin{theorem}
\label{theorem1}
Let $f$ be locally Lipschitz and suppose for every $x\in \mathbb R^n$, $f$ is either lower-$C^1$ or  upper-$C^1$ at $x$.
Let $x^1$ be  such that $Ax^1 \leq b$ and  $\{x\in \mathbb R^n: f(x)\leq f(x^1), Ax\leq b\}$
is bounded. 
Then every accumulation point $x^*$ of the sequence $x^j$ of serious iterates
generated by algorithm {\rm \ref{algo1}}
is a KKT-point of {\rm (\ref{program})}. 
\end{theorem}

\begin{proof}
The result will follow
from \cite[Theorem 1]{noll} as soon as we show
that downshifted tangents as modified in Example \ref{modified} and used in the algorithm
is a strict cutting plane oracle in the sense of definition \ref{strict-oracle}.
The remainder of the proof is to verify this.

1)
Let us denote cutting planes arising from the standard model $\phi^\sharp$
by $m_y^\sharp(\cdot,x)$,  cutting planes obtained by downshift as $m_y^\downarrow(\cdot,x)=t(\cdot)-s$,
and the true cutting plane of the oracle as $m_y(\cdot,x)$. Then as we know
$m_y(\cdot,x)=m_y^\downarrow(\cdot,x)$ if $m_y^\downarrow(y,x) \geq m_y^\sharp(y,x)$,
and otherwise $m_y(\cdot,x)=m_y^\sharp(\cdot,x)$. We have to check $(O_1)$, $(\widehat{O}_2)$, $(O_3)$.

2)
The validity of $(O_1)$ is clear, as both oracles provide Clarke tangent planes to $f$ at $x$ for 
$y=x$. 

3)
Let us now check $(O_3)$. Consider $x_j\to x$, and $y_j,y_j^+ \to y$. Here $y_j^+$ is a null step at serious step $x_j$.
Passing to a subsequence, we may distinguish  case I,
where $m_{y_j^+}(\cdot,x_j) = m_{y_j^+}^\sharp(\cdot,x_j)$ for every $j$,
and case II, where $m_{y_j^+}(\cdot,x_j)=m_{y_j^+}^\downarrow(\cdot,x_j)$ for every $j$.

Consider case I  first. Let $m_{y_j^+}^\sharp(y_j,x_j) = f(x_j) + g_j^\top (y_j-x_j)$, where $g_j\in \partial f(x_j)$
satisfies
$f^0(x_j,y_j^+-x_j)=g_j^\top (y_j^+-x_j)$. 
Passing to yet another subsequence, we may assume $g_j\to g$, and
upper semi-continuity of the Clarke subdifferential gives $g\in \partial f(x)$.
Therefore $m_{y_j^+}(y_j,x_j) = f(x_j) + g_j^\top (y_j-x_j)\to
f(x) + g^\top (y-x) \leq m_y^\sharp(y,x)\leq m_y(y,x)$. So here $(O_3)$ is satisfied
with $z=y$.

Newt consider case II. Here we have
$m_{y_j^+}(y_j,x_j) = t_{g_j}(y_j) - s_j$, where $t_{g_j}(\cdot)$ is a tangent to $f$ at $y_j^+$
with subgradient $g_j\in \partial f(y_j^+)$, and $s_j$ is the corresponding downshift
\[
s_j = \left[ t_{g_j}(x_j) -f(x_j) + c\|y_j^+-x_j\|^2   \right]_+.
\] 
Passing to a subsequence, we may assume $g_j\to g$, and by upper semi-continuity of $\partial f$ we have
$g\in \partial f(y)$. Therefore $s_j \to \left[ t_g(x) - f(x) + c\|y-x\|^2 \right]_+=:s$, where uniform
convergence $t_{g_j}(y_j)\to t_g(y)$ occurs due to the boundedness of $\partial f$. But now we see that
$s$ is the downshift for the pair $(x,y)$ when $g\in \partial f(y)$ is used.
Hence $m_{y_j^+}(y_j,x_j) \to m_y^\downarrow(y,x)$, and since
$m_y^\downarrow(y,x)\leq m_y(y,x)$, we are done. So again the $z$ in $(O_3)$ equals $y$ here.

4) Let us finally check axiom $(\widehat{O}_2)$. 
Let $x_j,y_j\to x$ be given.
We first consider the case when $f$ is upper-$C^1$ at $x$.
We have to find $\epsilon_j\to 0^+$
such that $f(y_j) \leq m_{y_j}(y_j,x_j) + \epsilon_j \|y_j-x_j\|$ as $j\to \infty$, and by the definition of the oracle,
it clearly suffices to show $f(y_j) \leq m_{y_j}^\sharp(y_j,x_j) + \epsilon_j\|y_j-x_j\|$.
By Spingarn \cite{spingarn},  or Daniilidis and Georgiev \cite{georgiev},
$-f$, which is lower-$C^1$ at $x$, has the following property: For every $\epsilon > 0$
there exists $\delta >0$ such that for all $0 < t < 1$ and $y,z\in B(x,\delta)$, 
\[
f(y) \leq f(z) + t^{-1} \left( f(z+t(y-z)) -f(z) \right) + \epsilon (1-t) \|z-y\|.
\]
Taking the limit superior $t\to 0^+$ implies
$$
f(y) \leq f(z) + f'(z,y-z) +\epsilon \|y-z\| \leq f(z) + f^0(z,y-z)+\epsilon \|y-z\|.
$$ 
Choosing $z=x_j$, $y=y_j$, $\delta_j = \|y_j-z_j\|\to 0$, we can find
$\epsilon_j \to 0^+$ such that
$f(y_j)\leq f(x_j) + f^0(x_j,y_j-x_j) + \epsilon_j \|y_j-x_j\|$, hence
$f(y_j) \leq m_{y_j}^\sharp(y_j,x_j) + \epsilon_j\|y_j-x_j\|$ by the definition of $m^\sharp_{y_j}(\cdot,x_j)$.
That settles the upper-$C^1$ case.

Now consider the case where $f$ is lower-$C^1$ at $x$. 
We have to find $\epsilon_j\to 0^+$
such that $f(y_j) \leq m_{y_j}(y_j,x_j) + \epsilon_j \|y_j-x_j\|$ as $j\to \infty$, and 
it suffices to show $f(y_j) \leq m_{y_j}^\downarrow(y_j,x_j) + \epsilon_j\|y_j-x_j\|$.
Since $m_{y_j}^\downarrow(y_j,x_j) \geq f(y_j) - s_j$, where 
$s_j$ is the downshift $s_j = \left[ t(x_j) - f(x_j) + c\|y_j-x_j\|^2 \right]_+$, and
$t(\cdot)=f(y_j) + g_j^\top (\cdot-y_j)$ for some $g_j\in \partial f(y_j)$, 
it suffices to exhibit $\epsilon_j\to 0^+$ such that
$f(y_j) \leq f(y_j)-s_j + \epsilon_j \|y_j-x_j\|$, or what is the same,
$s_j \leq \epsilon_j \|y_j-x_j\|$. For that it suffices to
arrange $\left[t(x_j) - f(x_j)\right]_+ \leq \epsilon_j \|y_j-x_j\|$, because once this is verified,
we get $s_j \leq \left[ t(x_j)-f(x_j) \right]_+ + c\|y_j-x_j\|^2
\leq (\epsilon_j + c\|y_j-x_j\|) \|y_j-x_j\| =: \widetilde{\epsilon}_j \|y_j-x_j\|$.
Note again that
by \cite{spingarn,georgiev} 
$f$ has the following property at $x$: For every $\epsilon > 0$
there exists $\delta > 0$ such that 
$f(tz+(1-t)y) \leq tf(z) + (1-t)f(y) + \epsilon t(1-t)\|z-y\|$ for all $y,z\in B(x,\delta)$.
Dividing by $t>0$  and passing to the limit $t\to 0^+$ gives
$f^0(y,z-y) \leq f(z) - f(y) + \epsilon \|y-z\|$, using the fact that $f$ is locally Lipschitz. 
But for every $g\in \partial f(y)$, $g^\top (z-y) \leq f^0(y,z-y)$.
Using $\|y_j-x_j\| =: \delta_j \to 0$ and taking $y=y_j$, $z=x_j$, this allows us to find
$\epsilon_j\to 0^+$ such that 
$g_j^\top (x_j-y_j) \leq f(x_j)-f(y_j)+\epsilon_j \|y_j-x_j\|$.
Substituting this above gives
$t(x_j)-f(x_j) = f(y_j)-f(x_j) + g_j^\top (x_j-y_j)\leq \epsilon_j \|y_j-x_j\|$ as desired.
That settles the lower-$C^1$ case.
\end{proof}

\section{Practical aspects of the algorithm}
\label{practical}
In this section we discuss several
technical aspects of the algorithm, which are
important for its performance.

\subsection{Stopping}
The stopping test in step 
2 of the algorithm is stated in this form for the sake of the convergence proof. In practice we delegate  stopping  
to the inner loop using
the following two-stage procedure.

If the inner loop 
at serious iterate $x^j$ finds the new serious step $x^{j+1}$
such that
\[
\frac{\|x^{j+1}-x^j\|}{1+\|x^j\|} < {\rm tol}_1, \quad \frac{|f(x^{j+1})-f(x^j)|}{1+|f(x^j)|} < {\rm tol}_2,
\]
then we decide that $x^{j+1}$ is optimal. In consequence, the $(j+1)$st inner loop will
not be executed. On the other hand, if the inner loop has difficulties  terminating
and  produces five consecutive
null steps $y^k$ where
\[
\frac{\|y^k-x^j\|}{1+\|x^j\|} < {\rm tol}_1, \quad \frac{|f(y^k)-f(x^j)|}{1+|f(x^j)|} < {\rm tol}_2,
\]
or if a maximum number $k_{\rm max}$  of allowed steps in the inner loop is reached, then
we decide that $x^j$ is optimal. In our experiments we use
\textcolor{black}{${\rm tol}_1= 10^{-5}$, ${\rm tol}_2= 10^{-5}$},
and $k_{\rm max}=50$.


\subsection{Recycling of planes}
At the beginning of a new inner loop at serious step $x^{j+1}$,
we do not want to start building the working model $\phi_1(\cdot,x^{j+1})$ from scratch.
It is more efficient to recycle some of the planes $(a,g) \in \mathcal G_{k_j}$
in the latest working model $\phi_{k_j}(\cdot,x^j)$. In the convex cutting plane
method, this is self-understood, as cutting planes are affine minorants of $f$,
and can at leisure stay on in the sets $\mathcal G$ at all times $j,k$. Without convexity,
we need the following recycling procedure:

Given a plane $m(\cdot,x^j) = a + g^\top (\cdot - x^j)$
in the latest set $\mathcal G_{k_j}$, we form the new downshifted plane
\[
m(\cdot,x^{j+1}) = m(\cdot,x^j) - s,
\]
where the downshift is organized as
\[
s = \left[  m(x^{j+1},x^j) - f(x^{j+1}) + c \|x^{j+1}-x^j\|^2\right]_+.
\]
In other words, we treat $m(\cdot,x^j)$ like a tangent to $f$ at null step
$x^j$ with respect to the serious step $x^{j+1}$ in the downshift oracle. We put
\[
m(\cdot,x^{j+1}) = a + g^\top (\cdot - x^j) - s
= a - s + g^\top (x^{j+1}-x^j) + g^\top (\cdot - x^{j+1}), 
\]
and
we accomodate $(a-s+g^\top(x^{j+1}-x^j),g)\in \mathcal G_1$
at the beginning of the $(j+1)$st inner loop.
\textcolor{black}{In the modified version we only keep a plane of this type in $\mathcal G_1$ after comparing it to
the exactness plane $m_0(\cdot,x^{j+1})= f(x^{j+1})+g^\top (\cdot - x^{j+1})$, $g\in \partial f(x^{j+1})$, which satisfies
$g^\top (x^{j}-x^{j+1})=f^0(x^{j+1},x^j-x^{j+1})$. Indeed, when $m(x^j,x^{j+1}) \geq m_0(x^j,x^{j+1})$, then we keep
the downshifted plane, otherwise we add $m_0(\cdot,x^{j+1})$ as additional exactness plane.  }

\section {The delamination benchmark problem} \label{delamination}

The interface behavior of laminated composite materials
is modeled by a non-monotone multi-valued function $\partial j$, 
characteristic of the  interlayer adhesive  placed at the contact boundary $\Gamma_c$.
In more precise terms, $\partial j$ is the physical law which holds
between the normal component $-S_n(s)|\Gamma_c$
of the  stress vector 
and the relative displacement $u_2(s)|\Gamma_c$, or {\em jump},  between the upper and lower boundaries. 
A typical law $\partial j$ for an interlayer adhesive is shown in Figure \ref{dellaw} (left).
In the material sciences, the knowledge of $\partial j$ is crucial
for the understanding of the basic failure modes of
the composite material. 

The adhesive law $\partial j$ is usually determined experimentally using
the double cantilever beam test \cite{gudladt} or other destructive testing methods. The result of a typical experiment
is shown schematically in Figure \ref{contaminated}  from \cite{gudladt}, where three probes with different levels of
contamination have been exposed. 
While the intact material shows stable
propagation of the crack front (dashed curve),
the 10\% contaminated
specimen shows a typical zig-zag profile (bold solid curve), indicating unstable crack front propagation. 
Indeed, when reaching the critical load $P=140$N, the crack starts to propagate. Since by the growth
of the  crack-elongation, 
the compliance of the structure increases, the crack propagation slows down and the crack is "caught", i.e.,  
stops at $u_2=0.25$mm and the load $P$ in the structure drops from $P=140$N to $P=40$N. 
Thereafter, due to the continuously increased load,  the crack starts
again to propagate until reaching another critical load
level at $P=90N$ and $u_2=5$mm. This phenomenon occurs five to six times, as seen in Figure \ref{contaminated}.

\begin{figure}[h!]
\centering
\includegraphics{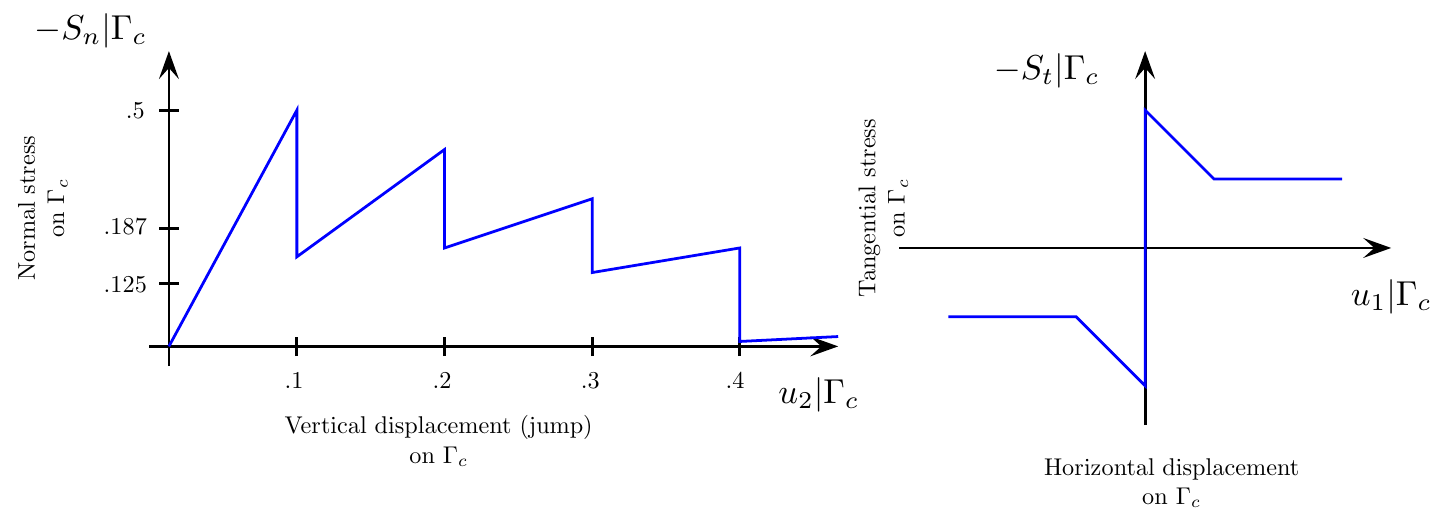}
\caption{Left image shows non-monotone delamination law $\partial j$, leading to an upper-$C^1$ objective. Right image
shows non-monotone friction law, leading to a lower-$C^1$ objective.  \label{dellaw}}
\end{figure}

The 50\% contaminated specimens (dotted curve) shows micro-cracks that appear at a finer level  and are
not visible in the Figure \ref{contaminated}. The lower level of the adhesive energy, which is represented by the area below the load-displacement curve, indicates now that this specimen is of minor resistance. 

Even though
the displacement $u_2$ in Figure \ref{contaminated} can only be measured
at the crack tip, in order to proceed one now {\em stipulates} the law $\partial j$ all along
$s\in \Gamma_c$  by assuming that the normal stresses 
$S_n(s)|\Gamma_c$
follow the  measured behavior
\begin{eqnarray}
\label{law1}
-S_n(s) \in  \partial j(s, u_2(s)), \; s\in \Gamma_c.
\end{eqnarray}
Under this hypothesis one now
solves the variational
inequality for the unknown displacement field ${\bf u}=(u_1,u_2)$, 
and then validates 
(\ref{law1}). 
Note that $S_n(s)|\Gamma_c$ is the {\em truly} relevant information, as it indicates the 
action of the destructive forces along $\Gamma_c$,  explaining
eventual failure of the composite. 
In current practice in the material sciences,
this information cannot be assessed by direct measurement,
and is therefore estimated by heuristic formulae \cite{gudladt}.  Our approach could  be
interpreted as one such estimation technique based on mathematical modeling.

\begin{figure}[h!]
\centering
\includegraphics[width =\textwidth]{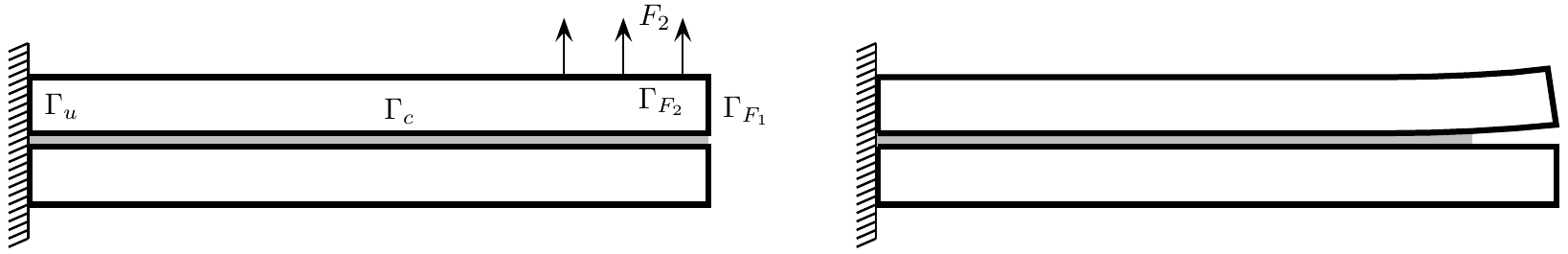}
\caption{Schematic view of cantilever beam testing. Under applied traction force $F_2$  the crack front
propagates to the left. In program (\ref{del2})  traction force $F_2$ and crack front length are given, while the corresponding
displacement $u$ and reactive forces $-S_n|\Gamma_c$ along the contact boundary $\Gamma_c$ have to be computed. \label{model}}
\end{figure}

 \subsection{Delamination study}
Within the framework of plane linear elasticity
we consider  a symmetric 
laminated structure with 
an interlayer adhesive under loading (see Fig. \ref{model}).
Because of the 
symmetry of the structure, it suffices to consider only the upper half of the specimen, represented by $\Omega \subset \mathbb R^2$.  
The Lipschitz boundary $\Gamma$ of $\Omega $ consists of four disjoint
parts $\Gamma_u$, $\Gamma_c$, $\Gamma_{F_1}$ and $\Gamma_{F_2}$. 
The body is fixed on $\Gamma_u$, i.e.,
$$
u_i=0 \; \mbox{on} \; \Gamma_u, \quad i=1,2.
$$
On $\Gamma_{{F_1}} $ the traction forces $\mathbf F$ are constant and given as
$$
\mathbf{F}=(0,F_2) \quad \mbox{on} \; \Gamma _{F_1}.
$$
The part $\Gamma _{F_2}$ is load-free.
We adopt standard notation from linear elasticity and introduce 
the bilinear form of  linear elasticity 
\begin{equation}\label{elastic}
 a(\mathbf u, \mathbf v) = \int _ \Omega \mathbf 
\varepsilon(\mathbf u) \, : \, \mathbf \sigma (\mathbf v) \, dx,
\end{equation}
where $\mathbf u=(u_1,u_2)$ is the displacement vector, 
$\mathbf \varepsilon ({\bf u}) = \frac{1}{2}(\nabla \mathbf u + (\nabla \mathbf u )^T)$ 
the linearized strain tensor,  and 
$\mathbf \sigma (\mathbf v) = \mathbf C : \mathbf \varepsilon (\mathbf v)$  the stress tensor.
 Here, $\mathbf C$ is the elasticity tensor with symmetric positive $L^{\infty}$
 coefficients. The bilinear form is symmetric and due to the first Korn inequality,
coercive. 
The linear form $\langle \mathbf{g}, \cdot \rangle$ is defined by
$$
\langle \mathbf{g},\mathbf{v}\rangle =F_2\int _{\Gamma_{F_1}}v_2\, ds.
$$
On the contact boundary $\Gamma_c$ we have the unilateral constraint
$$
u_2 \geq 0 \quad \mbox{a.e. on} \; \Gamma_c
$$
and we apply the non-monotone multi-valued adhesive  law
\begin{eqnarray}
\label{law}
 -S_n(s) \in \partial j(s, u_2(s)) \quad  \mbox{for a.a.} \; s\in \Gamma _c.
\end{eqnarray}
Here $S_n=\sigma_{ij} n_j n_i$, where  $\mathbf n=(n_1,n_2)$ is the outward unit normal vector to $\Gamma_c$.

A typical non-monotone law $\partial j(s, \cdot)$ for delamination, 
describing the behavior of the adhesive,  is shown in Fig. \ref{dellaw}. 
This law is derived from a nonconvex and a nonsmooth locally Lipschitz 
super-potential $j$  expressed in terms of a minimum function. 
In particular, $j(s,\cdot)$ is a minimum of four convex quadratic  and one linear function.  

We also assume that tangental traction can be neglected on $\Gamma_c$, i.e., $S_t(s)=0$. 
The weak
formulation of the delamination problem is then given by the following hemivariational 
inequality:   Find 
$\mathbf{u}\in K$ such that
\begin{equation} \label{del}
a(\mathbf{u},\mathbf{v}-\mathbf{u}) + \displaystyle \int_{\Gamma_c}j^0(s, u_2(s);v_2(s)-
u_2(s)) \, ds \geq \langle \mathbf{g}, \mathbf{v}-\mathbf {u}
\rangle \quad \forall \, \mathbf{v}\in K,
\end{equation}
where 
$j^0(s, u;d)$ is the  Clarke  directional derivative of 
$j(s,\cdot)$ at ${u}$ in direction $d$, 
$K$ is the nonempty, closed  convex set of all admissible displacements defined by
$$
K=\{\mathbf{v} \in V \, : \, v_2 \geq 0  \, \, \;  \mbox{on} \, \, \; \Gamma_c \},
$$
contained in the function space
$$
V=\{ \mathbf{v} \in H^1(\Omega; \mathbb R^2)\;: \; \mathbf{v}=0 \;\, \mbox{on} \;\, \Gamma _u \}.
$$
The potential energy  of the  problem is
$$
\Pi (\mathbf v) =\frac{1}{2} a(\mathbf v, \mathbf v) + J(\mathbf v)  - \langle \mathbf{g}, \mathbf{v} \rangle, 
$$
where $J:V\to \mathbb R$ defined by
$$
J(\mathbf v)=\int_{\Gamma_c} j(s, v_2(s)) \, ds 
$$
is the term responsible for the nonsmoothness. 
Using the potential energy, the hemivariational inequality 
  (\ref{del})  can be transformed to  the
following  nonsmooth, nonconvex
constrained optimization problem
of the form (\ref{program})
\begin{eqnarray}
\label{del2}
\begin{array}{ll}
\mbox{minimize} & \Pi(\mathbf u) \\
\mbox{subject to} & {\mathbf u} \in K
\end{array}
\end{eqnarray}
where the objective is upper-$C^1$, because the super-potential $j(s,\cdot)$ 
is a minimum. In particular, we have an objective of the form (\ref{structure}),
where the smooth part $f_s$ comprises $\frac{1}{2}a({\mathbf v},{\mathbf v})
- \langle \mathbf g,\mathbf v\rangle$, while the nonsmooth
part  $J(\mathbf v)=\int_{\Gamma_c} j(s,v_2(s))\,ds$ has the
form  (\ref{structure}) with a finite
index set $I$ once the boundary integral is suitably parametrized.

According to the existence theory in \cite{Naniewicz}, problem (\ref{del2})
has at least one Clarke critical point ${\mathbf u}^*$ satisfying
the necessary optimality condition
$$
0\in \partial \,\Pi ({\bf u}^*) + N_K({\bf u}^*),
$$
where $N_K({\mathbf u})$ is the normal cone to  
$K$ at $\mathbf u$,
and vice versa,
by a result in \cite{Miettinen}  every critical point of 
$\Pi$ on $K$ is a solution of (\ref{del}) (see also \cite{Makela}).

\begin{figure}[h!]
\includegraphics{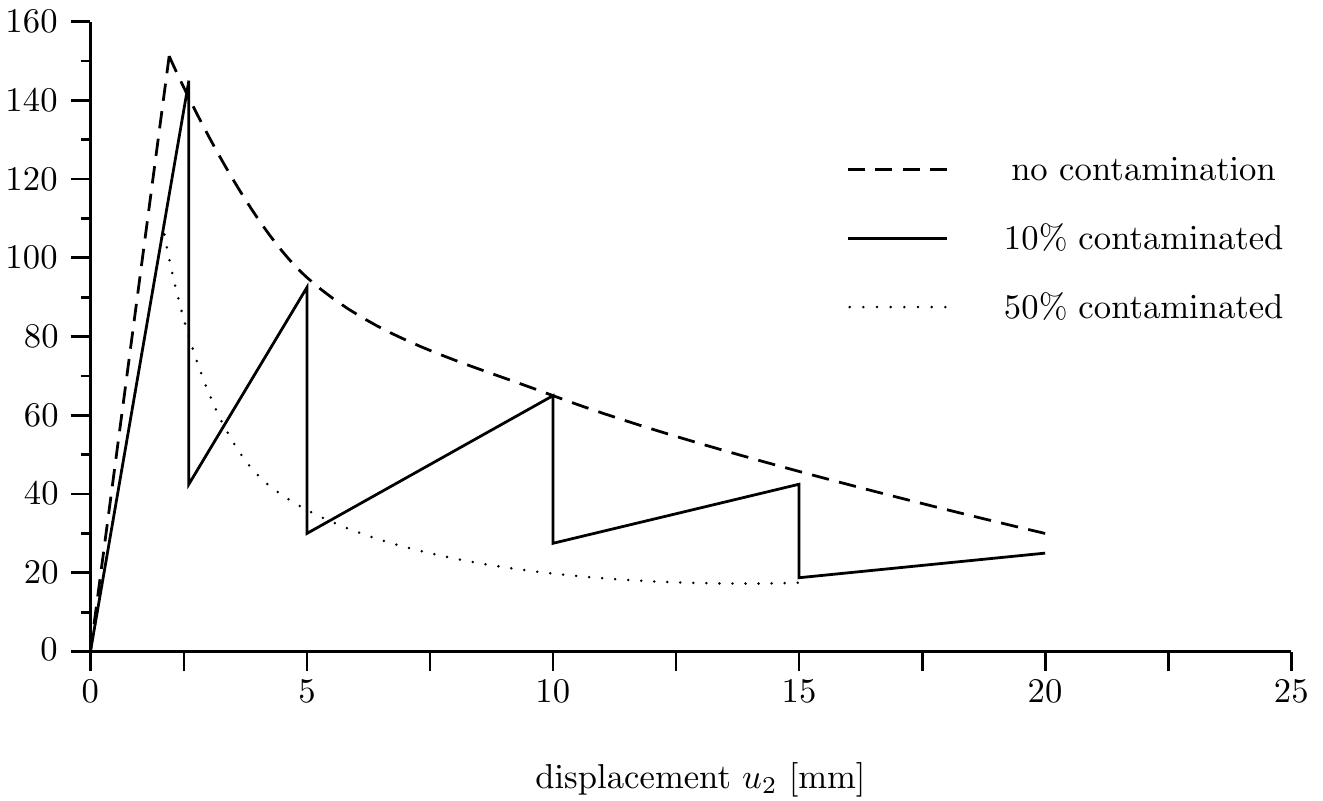}
\caption{Load-displacement curve determined by double cantilever beam test. Dashed curve shows stable behavior for
material without contamination. The 10\% contaminated specimen (bold solid curve) shows unstable crack growth. 
After initial linear growth, when the critical load $P=140$N is reached,
the crack starts to propagate. But then the propagation speed slows down, since by the crack the compliance of the specimen increases, and the crack is "caught" at $u_2=0.25$mm. The load $P$ drops from $P=140$N to $P=40$N. Then, by the constantly applied traction force, there is a linear growth of the load $P$  from $ P=40$N to the critical load $P=90$N, where the crack propagates again and stops at $u_2=5$mm, with the load now reduced to $P=30$N.  
The 50\% contaminated specimen exhibits  micro-cracks not visible at the chosen scale. 
\label{contaminated}}

\end{figure}

\subsection{Discrete problem}
We consider a regular triangulation $\{\mathcal T_h\}$ of $\Omega$, where
we first divide $\Omega$ into small squares of  size $h$ and then each
square by its diagonal into two triangles. 
To approximate 
$V$ and $K$  we use a piecewise linear finite element approximation and set 
$$
V_h=\{v_h \in C(\overline{\Omega};\mathbb R^2)\, : \, {v_h}_{|_{T}} \in ({\bf P}_1)^2,
\, \forall \,T \in \mathcal T_h, \, {v_h}_{|_{\Gamma_u}}=0 \},
$$
$$
K_h=\{v_h \in V_h: v_{h2}(s_\nu)\geq 0 \quad \forall \, s_\nu \in \overline{\Gamma}_c\backslash \overline{\Gamma}_u\}.
$$
Similar to low order finite element approximations of nonsmooth convex contact problems \cite{Glowinski1, Gwinner}, 
we use the trapezoidal quadrature rule to approximate the functional $J$ by 
 \begin{eqnarray}
\label{Jh}
J_h(v_h)  =  \frac{1}{2} \displaystyle \sum_{s_\nu\in \overline{\Gamma}_c\setminus\overline{\Gamma}_u}  |s_\nu s_{\nu+1}| \big [ 
j(s_\nu,v_{h2}(s_\nu)) + 
 j(s_{\nu+1},v_{h2}(s_{\nu+1}))
\big ] ,  
\end{eqnarray}
where we  are summing over the 
nodes  $s_\nu$ on the contact boundary $\overline{\Gamma}_c\backslash  \overline{\Gamma}_u$, with
$s_{\nu+1}$ being the neighbor of node $s_\nu$ on $\Gamma_c$ in the sense of integration.
This can be regrouped as
\[
J_h(v_h)= \sum_{s_\nu\in \overline{\Gamma}_c\setminus\overline{\Gamma}_u} c_\nu
j(s_\nu,v_{h2}(s_\nu)) = \sum_{s_\nu\in \overline{\Gamma}_c\setminus\overline{\Gamma}_u} c_\nu \min_{i\in  I}
j_i(s_\nu,v_{h2}(s_\nu))
\]
with appropriate weights $c_\nu>0$.
Here, $ I$ is the set of zig-zags in the graph of $\partial j$. 

The bundle  algorithm is applied to minimize the discrete functional
\begin{equation} \label{discreteJ}
\Pi_h(v_h)= \frac{1}{2} a(v_h,v_h)+J_h(v_h) - \langle g, v_h \rangle \quad \mbox{on} \quad K_h.
\end{equation}
Introducing an index set $N$  
for the nodes $s_\nu$ on the contact boundary $\overline{\Gamma}_c$, we may pull out the minimum from
under the sum, which leads to the expression
\[
\Pi_h(v_h)=\frac{1}{2} a(v_h,v_h)+
\min_{i(\cdot) \in I^N} \sum_{\nu\in N} c_\nu j_{i(\nu)}(s_\nu,v_{h2}(s_\nu)) - \langle g,v_h\rangle.
\]
This is the discrete version of  (\ref{struct}), where 
$\frac{1}{2}a(v_h,v_h)-\langle g,v_h\rangle$ is the smooth term $f_s$, and $J_h$ the nonsmooth part. 

While computation of Clarke subgradients is straightforward here,
we still have to explain how the
matrix $Q=Q(v)$ in the second-order working model (\ref{second}) is chosen. 
Discretizing the quadratic form of linear elasticity 
as $a(v_h,v_h)=v_h^\top {\bf A} v_h$ with the symmetric stiffness matrix ${\bf A}$,
and observing that $\langle g,v_h\rangle = {\bf g}^\top v_h$ is linear,
we choose $Q(v)  ={\bf A} + \sum_{\nu\in N} \nabla^2 j_{i(\nu)} (s_\nu,v_h(s_\nu))$, where 
$i(\nu)\in I$ is one of those indices, where the minimum
$\min_{i\in I} j_i(s_\nu,v_{h2}(s_\nu))$ is attained. 

For convergence of the lowest-order finite element approximation used here
we refer to the results in \cite{Ovcharova}. Higher-order approximations with no limitation in the
polynomial degree, 
which lead to nonconforming approximation of unilateral constraints, 
have only recently been analyzed for monotone contact problems, see \cite{Gwinner_2013}.

\subsection{Numerical results}
We present numerical results obtained in a  delamination simulation
with modulus of elasticity $E=210$ GPa and Poisson ratio $\nu =0.3$ corresponding to a steel specimen. In all examples we use the benchmark model of \cite{Baniotopoulos2} with geometrical characteristics $(0,100) \times (0,10) $ in [mm] and
thickness $5$mm.  We apply our bundle method to
(\ref{del2}) and compare the results to those obtained by
the regularization technique in \cite{Ovcharova, Ovcharova_Gwinner}. All  computations use 
piecewise linear functions and the discretization $40 \times 4$ corresponding to $h=0.25$cm. 
In this case, the number of the unknowns in the discrete problem (\ref{discreteJ}) is $80$. 

\begin{figure}[h!]
\includegraphics[width=7.8cm,height=4.6cm]{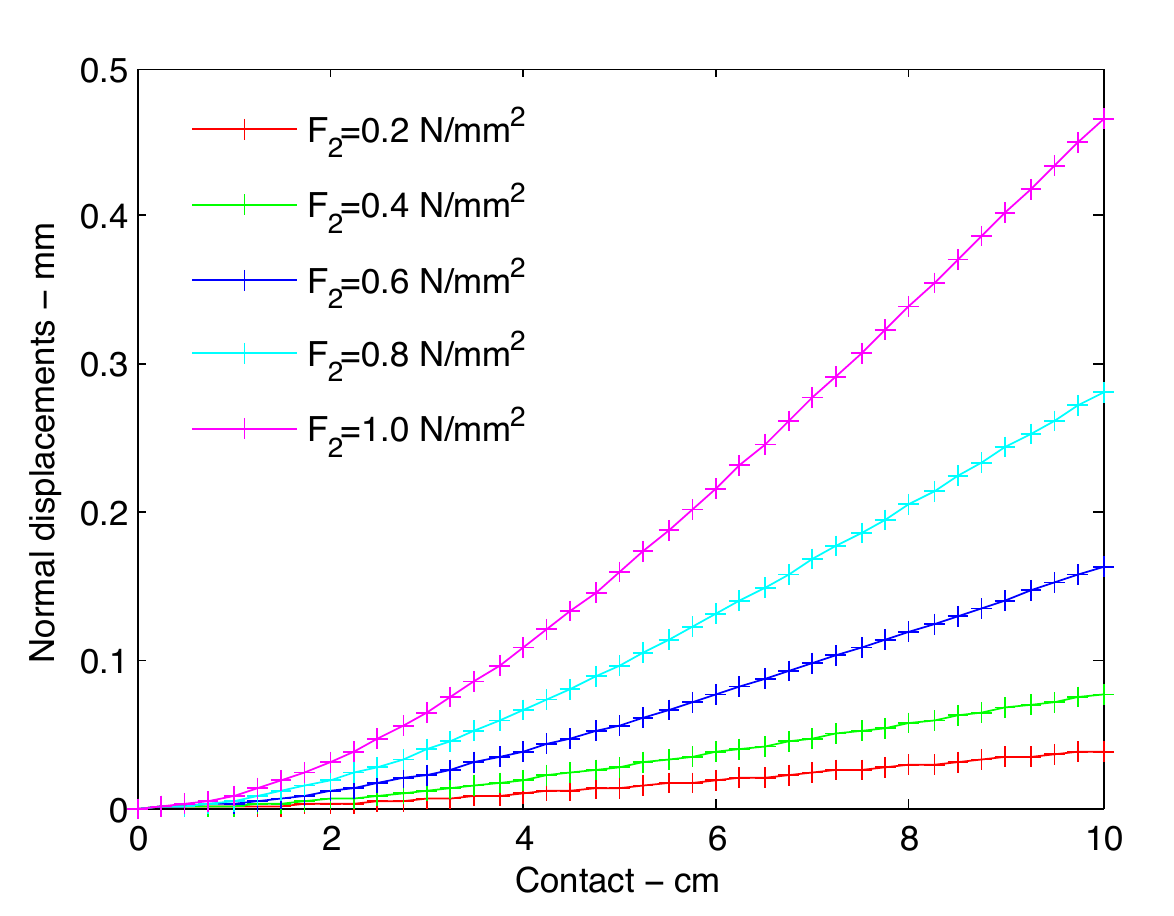}
\includegraphics[width=7.8cm,height=4.6cm]{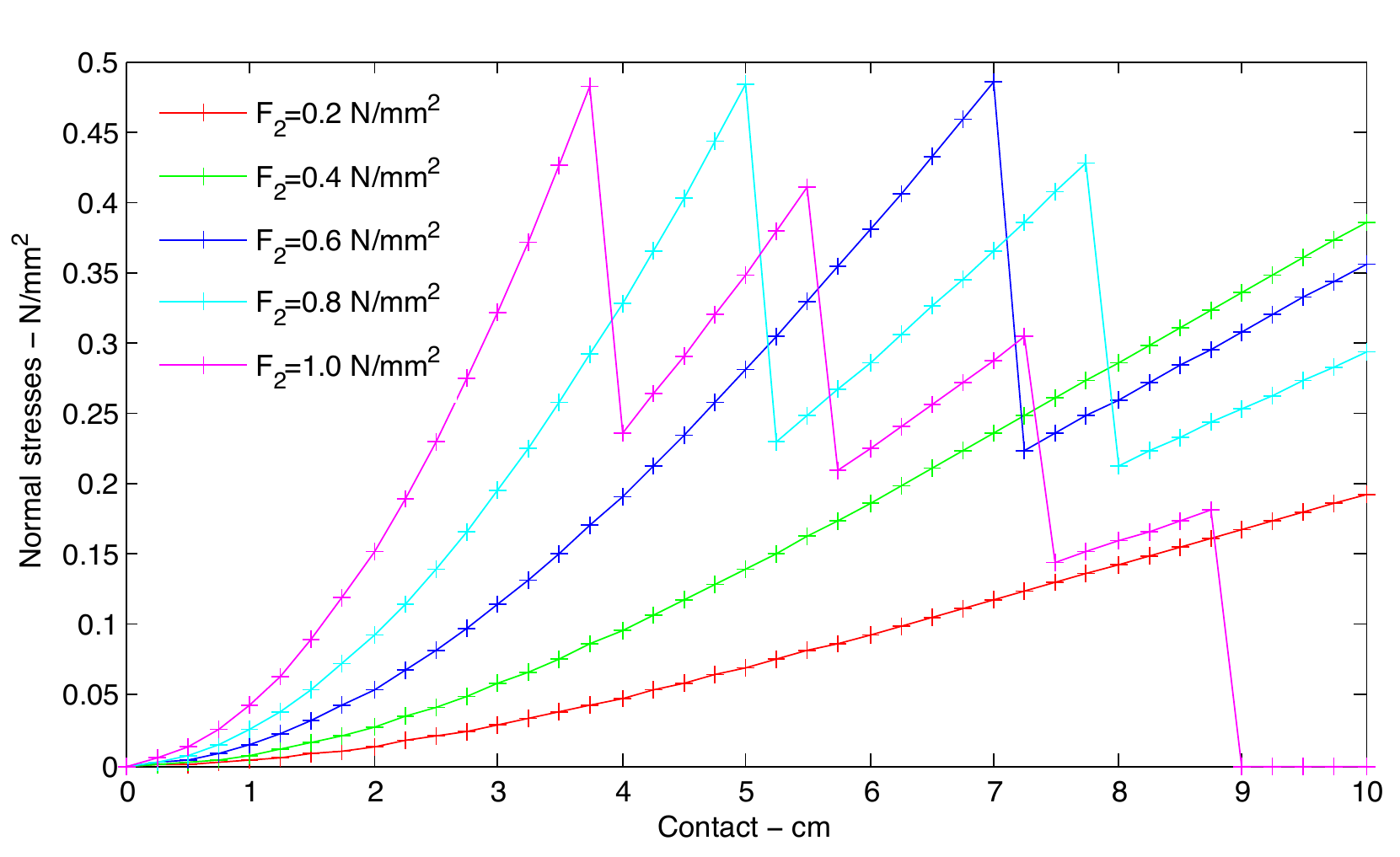}
\includegraphics[width=7.8cm,height=4.7cm]{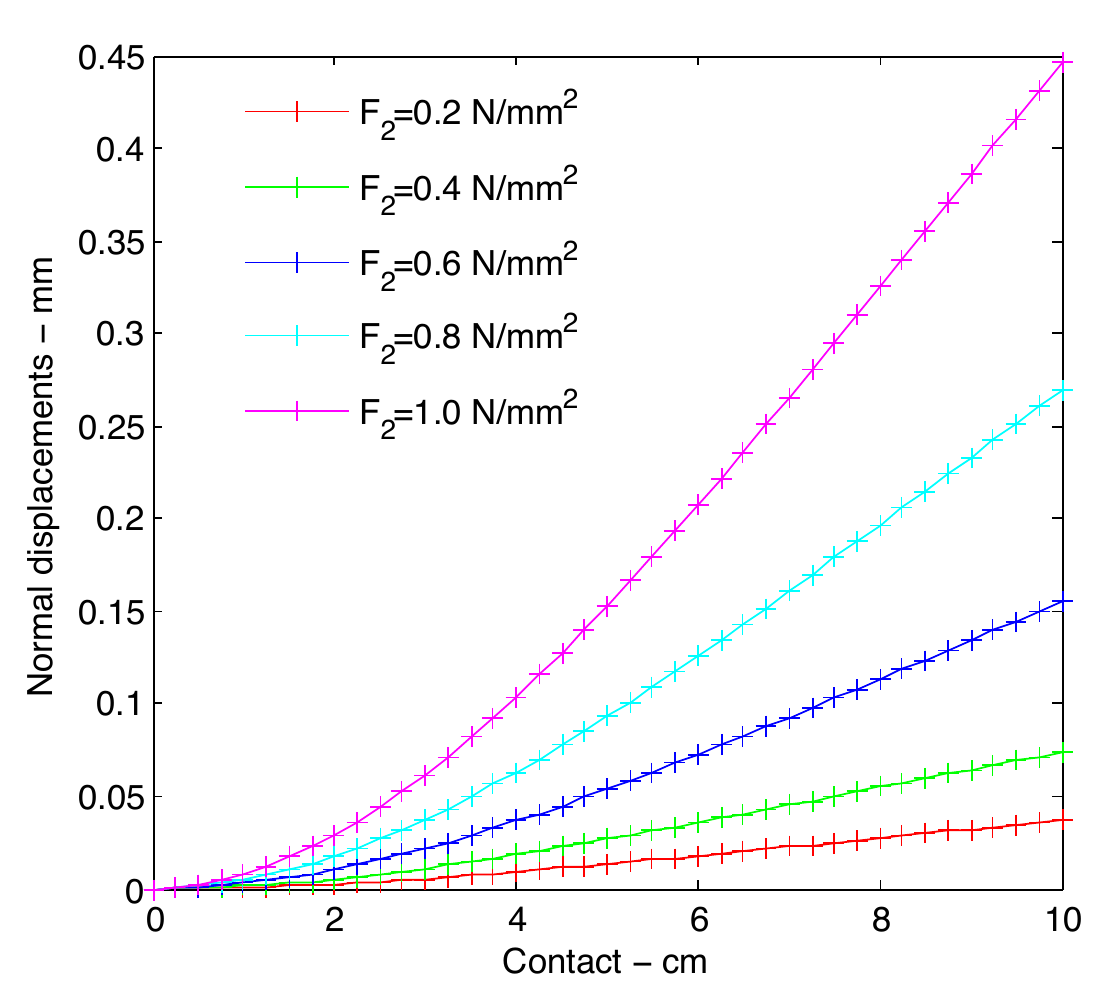}
\includegraphics[width=7.8cm,height=4.6cm]{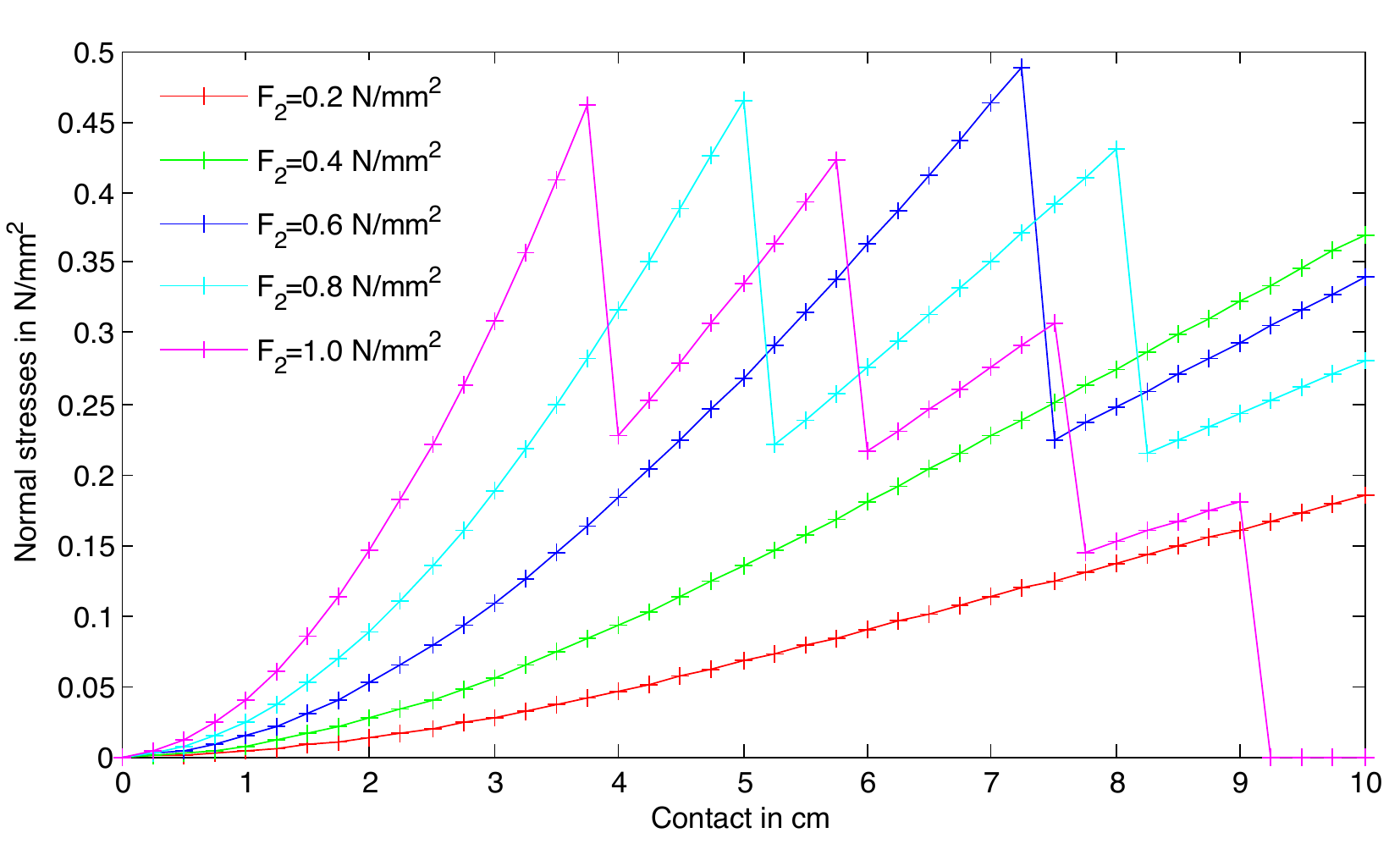}
\caption{Upper: regularization method of  \cite{Ovcharova, Ovcharova_Gwinner}. Lower: optimization method. Left image shows vertical displacement $u_2$
for 5 different values of  $F_2$. Right image shows vertical component of reactive force along contact boundary for   same 5
scenarios.}
\end{figure}

\begin{table}[h!]
\caption{Regularization. Vertical displacement [mm] at 4 intermediate points  for  same 5 scenarios.}
\centering
\begin{tabular}{||c||c|c|c|c||}
\hline\hline
$F_2 [N/mm^2]$  & $u_2(x_1)$ & $u_2(x_2)$ & $u_2(x_3)$ & $u_2(x_4)$ \\
\hline\hline
0.2&4.154500e-06&1.394500e-05&2.601700e-05&3.858700e-05\\
0.4&8.308100e-06&2.788800e-05&5.202800e-05&7.716600e-05\\
0.6&1.633200e-05&5.622700e-05&1.080000e-04&1.640000e-04\\
0.8&2.792500e-05&9.663100e-05&1.860000e-04&2.810000e-04\\
1.0&4.600600e-05&1.590000e-04&3.080000e-04&4.660000e-04\\
\hline\hline
\end{tabular}
\end{table}

\begin{table}[h!]
\caption{Optimization. Vertical displacement [mm] at four intermediate  points for  same 5 scenarios.}
\centering
\begin{tabular}{||c||c|c|c|c||}
\hline\hline
$F_2 [N/mm^2]$  & $u_2(x_1)$ & $u_2(x_2)$ & $u_2(x_3)$ & $u_2(x_4)$ \\
\hline\hline
0.2&4.022500e-06&1.345400e-05&2.499300e-05&3.691900e-05\\
0.4&8.069300e-06&2.698800e-05&5.013300e-05&7.404900e-05\\
0.6&1.564800e-05&5.373900e-05&1.030000e-04&1.550000e-04\\
0.8&2.691300e-05&9.297200e-05&1.790000e-04&2.700000e-04\\
1.0&4.414000e-05&1.530000e-04&2.940000e-04&4.470000e-04\\
\hline\hline
\end{tabular}
\end{table}

\begin{table}[h!]
\caption{Regularization. Horizontal displacement [mm] at four intermediate points  for  same 5 scenarios.}
\centering
\begin{tabular}{||c||c|c|c|c||}
\hline\hline
$F_2 [N/mm^2]$  & $u_2(x_1)$ & $u_2(x_2)$ & $u_2(x_3)$ & $u_2(x_4)$ \\
\hline\hline
0.2&1.481900e-06&2.251300e-06&2.474400e-06&2.499500e-06\\
0.4&2.963600e-06&4.502200e-06&4.948300e-06&4.998500e-06\\
0.6&5.918500e-06&9.400600e-06&1.077100e-05&1.097500e-05\\
0.8&1.015200e-05&1.625600e-05&1.866400e-05&1.904000e-05\\
1.0&1.674400e-05&2.690100e-05&3.100500e-05&3.167000e-05\\
\hline\hline
\end{tabular}
\end{table}

\begin{table}[h!]
\caption{Optimization. Horizontal displacement [mm] at four intermediate points  for  same 5 scenarios.}
\centering
\begin{tabular}{||c||c|c|c|c||}
\hline\hline
$F_2 [N/mm^2]$  & $u_2(x_1)$ & $u_2(x_2)$ & $u_2(x_3)$ & $u_2(x_4)$ \\
\hline\hline
0.2&1.432200e-06&2.161500e-06&2.356100e-06&2.368400e-06\\
0.4&2.872700e-06&4.335000e-06&4.724700e-06&4.748800e-06\\
0.6&5.663400e-06&8.957000e-06&1.023200e-05&1.041100e-05\\
0.8&9.777300e-06&1.561000e-05&1.787700e-05&1.822600e-05\\
1.0&1.606400e-05&2.578000e-05&2.970700e-05&3.034700e-05\\
\hline\hline
\end{tabular}
\end{table}

\begin{figure}[ht]
\hspace*{-.4cm}
\includegraphics[width=7.8cm,height=4.7cm]{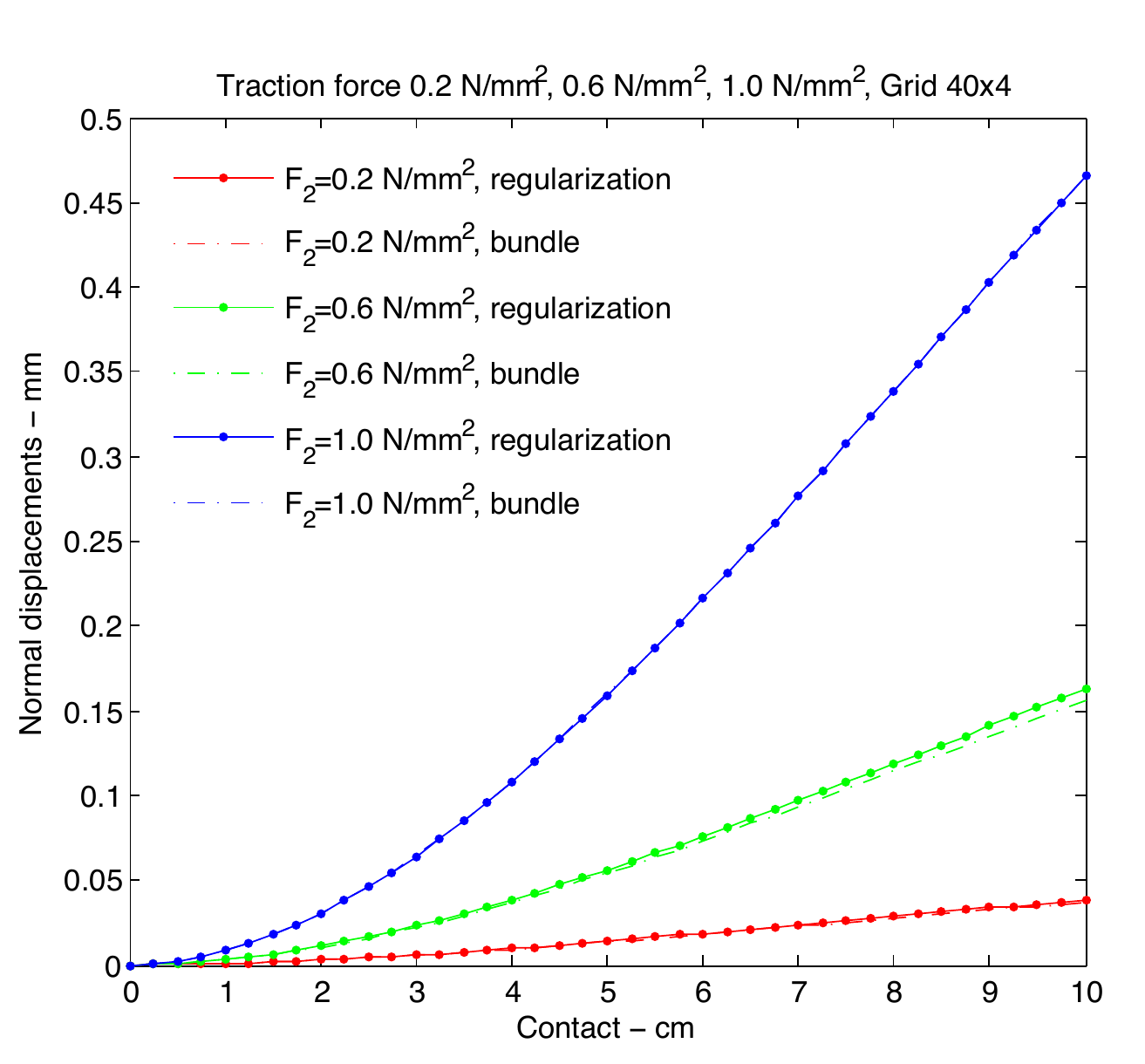}
\includegraphics[width=7.8cm,height=4.7cm]{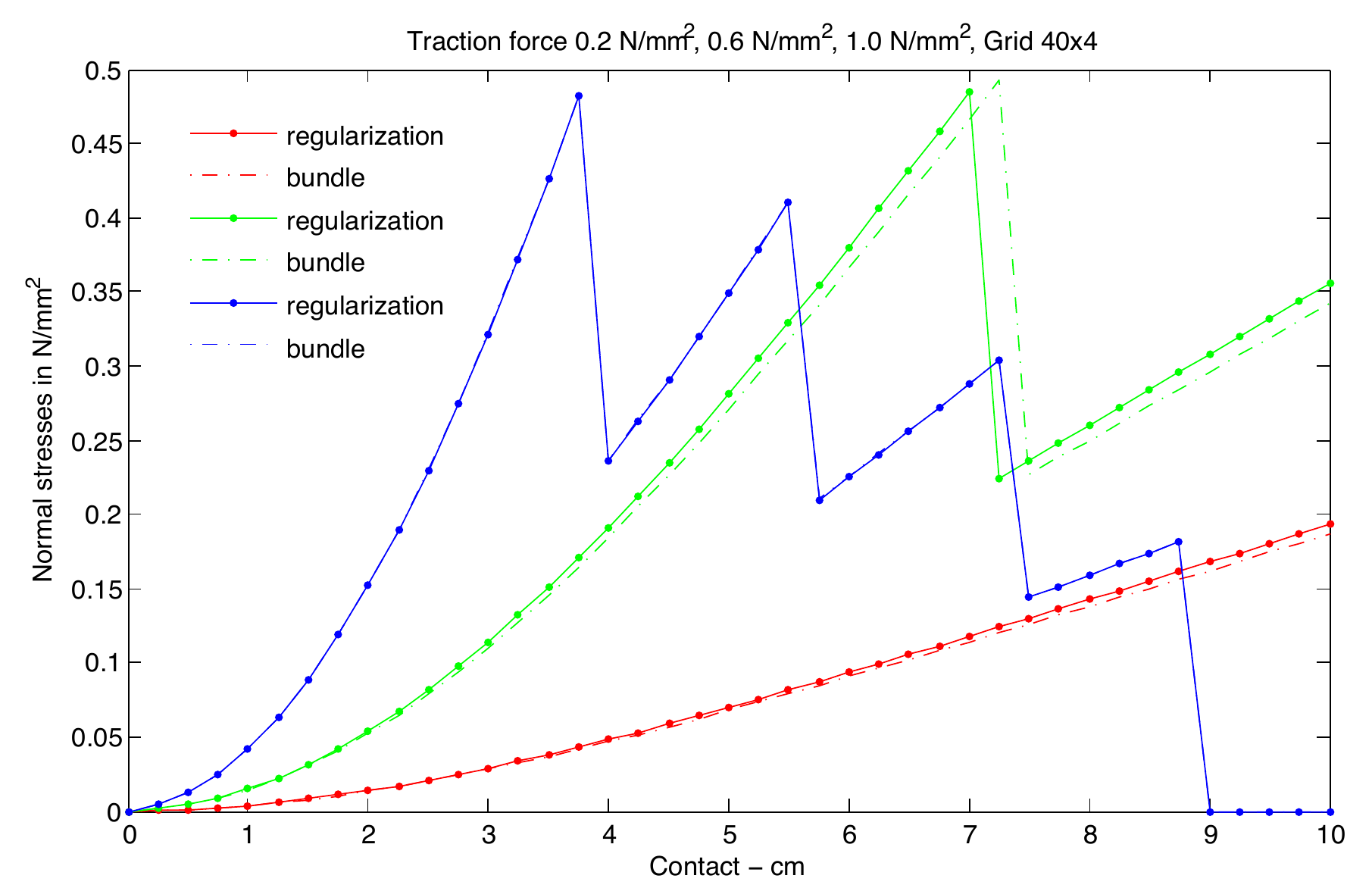}
\caption{Comparison of regularization (bold solid curves) and optimization (dashed) for 3 different values of  $F_2$.
Left vertical displacement, right reactive force.}
\end{figure}

\begin{table}[h!]
\caption{Comparison of optimal valued obtained by regularization and optimization}
\centering
\begin{tabular}{||c|c|c||}
\hline\hline
$F_2 [N/m^2] $             &    ${\Pi_h}_{\rm reg} [Nm]$    &     ${\Pi_h}_{\rm opt}$ [Nm] \\
\hline\hline
200000 &	-1.32894	&-1.29271 \\
\hline	
400000	&-2.35224&	-2.30025\\	
\hline
600000&	-3.83972&	-3.74609 \\	
\hline
800000	&-5.08164	&-5.05389	\\
\hline
1000000&	-5.66771	&-5.66770	\\
\hline\hline
\end{tabular}
\end{table}

\section*{Conclusion}
We have presented a bundle method based on the mechanism of
downshifted tangents which is suited to optimize upper- and lower-$C^1$ functions. Our method allows to integrate
second-order information, if available, and gives a convergence certificate in the sense of subsequences. Every accumulation point of the sequence of serious iterates with an arbitrary starting point is critical.
We have successfully applied our method to a delamination problem arising in the material sciences, where upper-$C^1$ functions
have to be minimized. 
Results obtained by optimization were compared to results obtained by  the regularization technique of
 \cite{Ovcharova, Ovcharova_Gwinner},  and both methods are in good agreement.

\section*{Acknowledgment} 
The authors thank H.-J. Gudladt for many useful discussions. 
The authors were  partially supported by Bayerisch-Franz\"osisches Hochschulzentrum (BFHZ).

\end{document}